\documentclass[final,a4paper]{amsart}

\title{Schatten classes for Hilbert modules over commutative $C^*$-algebras}
\author{Abel B. Stern \and Walter D. van Suijlekom}

  \address{
Institute for Mathematics, Astrophysics and Particle Physics, Radboud
University Nijmegen, Heyendaalseweg 135, 6525 AJ Nijmegen, The Netherlands.
}

\usepackage[final]{microtype}
\usepackage{amsmath,amssymb,amsfonts,amsthm}
\usepackage{braket}
\usepackage{todonotes}
\usepackage{enumitem}
\usepackage{hyperref}

\hypersetup{pdfauthor={Abel B. Stern and Walter D. van Suijlekom},pdftitle={Schatten classes for Hilbert modules over commutative C*-algebras}}

\usepackage{mathtools}
\usepackage{ragged2e}

\DeclareSymbolFont{bbold}{U}{bbold}{m}{n}
\DeclareSymbolFontAlphabet{\mathbbold}{bbold}

\newcommand{\strongly}{$*$-strongly}
\newcommand{\unitization}[1]{{#1}^+}
\newcommand{\indicator}[1]{\mathbbold{1}_{#1}}
\newcommand{\cstar}{$C^*$}
\newcommand{\C}{\mathbb{C}}
\newcommand{\R}{\mathbb{R}}
\newcommand{\N}{\mathbb{N}}
\newcommand{\Z}{\mathbb{Z}}
\newcommand{\ip}[1]{\left\langle #1 \right\rangle}
\newcommand{\ketbra}[2]{\ket{#1} \bra{#2}}
\newcommand{\fram}{\mathfrak{e}}
\newcommand{\framt}{\mathfrak{f}}

\newcommand{\eqdef}{\coloneqq}%{}%{\overset{\underset{\mathrm{def}}{}}{=}}
\newcommand{\lnmp}[1]{\mathcal{L} \left( #1 \right)}
\newcommand{\cptmp}[1]{\mathcal{K} \left( #1 \right)}

\newcommand{\eqex}{\overset{\underset{!}{}}{=}}
\newcommand{\sch}[2]{\mathcal{L}^{#1} \left( #2 \right)}
\newcommand{\norm}[1]{\left\lVert #1 \right\lVert}

\newcommand{\cbstr}{C^{\mathrm{str}}_\mathrm{b}}

\newcommand{\gfdual}[1]{\widehat{#1}}
\newcommand{\Del}{\Delta}

\newcommand{\spinc}{spin\textsuperscript{c}}

\DeclareMathOperator{\res}{Res}
\DeclareMathOperator{\End}{End}
\DeclareMathOperator{\id}{id}
\DeclareMathOperator{\tr}{tr}
\DeclareMathOperator{\ran}{ran}

\DeclareMathOperator{\mat}{Mat}

\DeclareMathOperator{\dom}{dom}

\newtheorem{theorem}{Theorem}[section]
\newtheorem{lemma}[theorem]{Lemma}
\newtheorem{proposition}[theorem]{Proposition}
\newtheorem{corollary}[theorem]{Corollary}

\theoremstyle{plain}
\newtheorem{example}[theorem]{Example}
\newtheorem{remark}[theorem]{Remark}

\theoremstyle{plain}
\newtheorem{definition}[theorem]{Definition}
\newtheorem{assumption}[theorem]{Assumption}

\newcommand{\nocontentsline}[3]{}
\newcommand{\KeepFromToc}[1]{\bgroup\let\addcontentsline=\nocontentsline#1\egroup}

\newcommand{\myacknowledgments}[1]{%
\KeepFromToc{%
\section*{Acknowledgments}
#1}%
}

\newcommand{\nwoacknowledgment}{
\vspace*{\baselineskip}
\noindent
\begin{minipage}[t][15mm]{11mm}
\includegraphics[height=15mm]{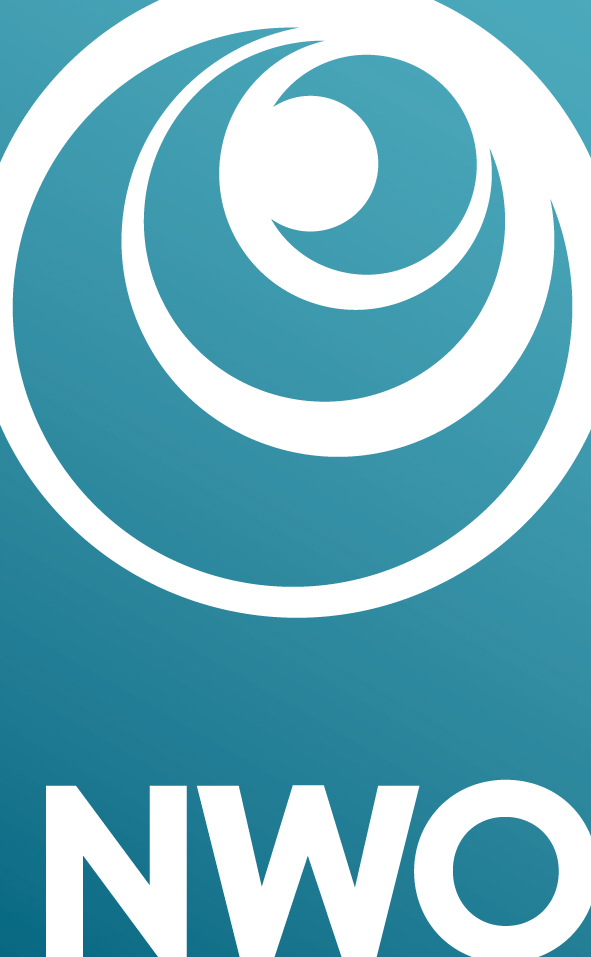}
\end{minipage}\hfill%
\begin{minipage}[b][15mm][c]{\linewidth-11mm}
This work is part of the research programme of the Foundation for Fundamental Research on Matter (FOM), which is part of the Dutch Research Council (NWO).
\end{minipage}
}

\makeatletter
\DeclareMathOperator{\exterior}{\@ifnextchar^\@exterior{\@exterior^{}}}
\def\@exterior^#1{\mathop{\bigwedge\nolimits^{\!#1}}}
\makeatother

\usepackage[backend=biber, giveninits=true, maxnames=5, useprefix=false, sorting=nyt]{biblatex}
\AtEveryBibitem{%
  \clearfield{issn} % Remove issn
  \clearfield{doi} % Remove doi

  \ifentrytype{online}{}{% Remove url except for @online and @softrware
	\ifentrytype{software}{}{
	    \clearfield{url}
	  }
	}
}

\begin{document}

\begin{abstract}
We define Schatten classes of adjointable operators on Hilbert modules over abelian \cstar{}-algebras.
Many key features carry over from the Hilbert space case. In particular, the Schatten classes form two-sided ideals of compact operators and are equipped with a Banach norm and a \cstar{}-valued trace with the expected properties.
For trivial Hilbert bundles, we show that our Schatten-class operators can be identified bijectively with Schatten-norm--continuous maps from the base space into the Schatten classes on the Hilbert space fiber, with the fiberwise trace.
As applications, we introduce the \cstar{}-valued Fredholm determinant and operator zeta functions, and propose a notion of $p$-summable unbounded Kasparov cycles in the commutative setting.
\end{abstract}

\maketitle
\tableofcontents

\section{Introduction}

The trace is a fundamental and highly versatile invariant of operators on Hilbert spaces.
In many applications, however, one is rather concerned with continuous \emph{families} of such operators.
From the perspective of Gelfand duality, the natural framework for such continuous families is that of Hilbert \cstar-modules over an abelian base.
The present study provides a systematic construction of trace and Schatten classes in this setting.
We establish some key properties for them that are familiar from the Hilbert space case.
We also consider some applications of the general theory.

\subsubsection*{The finite-rank trace}

The $\ast$-algebra $M(A)$ of finite matrices over a \cstar-algebra $A$ comes naturally equipped with a positive linear map
\[
\tr \colon M(A) \to A,\ (a_{ij})_{ij} \to \sum_i a_{ii},
\]
which clearly commutes with the entrywise lift of linear maps $A \to B$
and is cyclic if and only if the algebra $A$ is commutative.

If $E$ is a finitely generated projective Hilbert $A$-module, any compact adjointable endomorphism of $E$ can be represented as an element of $M(A)$ by a choice of isomorphism $E \oplus F \simeq A^n$.
Whenever $A$ is commutative, cyclicity implies that the trace of the resulting matrix is invariant under the choice of isomorphism and so 
constitutes the $A$-valued (Hattori-Stallings) trace on $\End_A (E)$.

Chapter \ref{chapter:abelian-schatten-classes}, below, introduces a robust framework that generalizes the construction to \emph{countably} generated Hilbert modules.
% From a topological perspective, the Serre-Swan theorem translates the previous analysis to the perhaps more obvious statement that there is a unique fiberwise trace map $\tr \colon \Gamma(\End V) \to C_0(X)$ of continuous endomorphisms of (finite-rank) Hermitian vector bundles $V \to X$ over a compact Hausdorff base $X$, such that $(\tr \phi)(x) = \tr \phi|_{V_x}$ for all $x \in X$.

% , or equivalently to generalize the setting of finite-rank Hermitian vector bundles to that of continuous fields of Hilbert spaces.

\subsubsection*{Continuous families of Schatten-class operators}

Under Gelfand duality and the Serre-Swan theorem, the finite-rank trace is just the fiberwise trace $\tr \colon \Gamma(\End V) \to C_0(X)$ of continuous endomorphisms of finite-rank Hermitian vector bundles over a locally compact Hausdorff space $X$.

If $V \to X$ is instead a continuous field of separable Hilbert spaces, then there is still a trace map on the fiberwise trace classes.
The challenge is then to unify these fiberwise trace classes in such a way as to yield a $C_0(X)$-valued fiberwise trace that retains the fundamental properties of the trace class on a fixed Hilbert space $H$.

As a fundamental example, consider the trivial bundle $H \times X \to X$ with fiber a fixed separable Hilbert space $H$.
The \cstar-algebra of its adjointable endomorphisms consists of all \strongly{} continuous, bounded families of bounded operators on $H$.
One wonders whether this algebra, denoted $\cbstr(X, B(H))$, contains two-sided ideals of ``continuous Schatten-class operators'' such that some or most of the usual theory of Schatten classes on Hilbert spaces is preserved.

In order to ensure continuity of the trace, the least one should demand of such an ideal is that the pointwise Schatten norms lie in $C_0(X)$.
On the other hand, the strongest reasonable condition at hand is that the families are themselves Schatten-norm continuous, that is, that they lie in the Banach space $C_0(X, \sch{p}{H})$.
Through careful control over the relation between the Schatten classes on the standard Hilbert space $l^2(\C)$ and the complex matrix algebras $M_n(\C)$, we are able in Theorem~\ref{theorem:lp-iff-tr-tp-in-c0x} to show that these conditions do in fact coincide on $\cbstr(X, B(H))$ and yield a two-sided ideal that is contained in the compact operators and is closed under its Banach norm. 

Kasparov's stabilization theorem and unitary invariance of the Schatten norms allow us to easily generalize the trivial bundle example to all continuous fields of separable Hilbert spaces in Theorem~\ref{theorem:schp-is-ideal-and-norm-has-properties}.
A further upshot of this approach is that much of the pointwise Schatten-class theory, including the Hölder-von Neumann inequality, carries over easily to the general case, cf. Theorem~\ref{thm:norm-prop-Lp}.

\subsubsection*{Frames and the fiberwise trace}

The theory of Schatten classes on Hilbert spaces is often mediated through the language of orthonormal bases and diagonalization. The approach of Section~\ref{section:schatten-class-standard-module} shows that one may very well work with (standard normalized) \emph{frames} \cite{MR1938798,MR1949886} instead. This allows for straightforward generalization of the familiar formulas to the theory of Hilbert modules, and indeed, the result is what one would hope for: the fiberwise trace turns out to be the norm-convergent sum over the diagonal in a given frame, cf. Theorem~\ref{theorem:nonstandard-schatten-class-is-pullback-by-frame}.
In the context of frames, it was earlier remarked in \cite{MR1332446},\cite[Proposition 4.8]{MR1938798} that the obvious notion of Hilbert-Schmidt inner product is invariant both under the choice of frame and under the adjoint. That observation is supplied with the necessary context as a special case of our Schatten-class operators in Section~\ref{section:hilbert-schmidt}.

\subsubsection*{Applications}
By the same principles as for the finite-rank trace, \cite{MR424786} introduced, in the context of K-theory, the \emph{Fredholm determinant} of endomorphisms of finitely generated modules over unital commutative algebras. As the Fredholm determinant is interesting in its own right, Section~\ref{section:fredholm-determinant} uses the result on the Schatten classes to extend its definition and basic properties to the setting of countably generated modules over unital commutative \cstar-algebras. A straightforward generalization of \cite{MR432738} remains however elusive, due to the conceptual problems in generalizing the relevant category.

Spectral geometry is the study of Riemannian manifolds $M$ via the spectra of differential operators, such as \spinc{}-Dirac operators $D$, on $M$. An important example of spectral invariant is the localized heat trace $t \mapsto \tr f e^{-t D^2}, f \in C(M)$. It determines the volume and total scalar curvature of $M$, is strongly related to the Atiyah-Singer Index Theorem \cite{MR1396308}, and is able to describe classical field theories on $M$ through the spectral action principle \cite[Chapter 11]{MR2371808}.
The first step in generalizing the above to unbounded Kasparov cycles, that is, certain $(C_0(M), C_0(N))$-Hilbert bimodules carrying a selfadjoint, regular, $C_0(N)$-linear unbounded operator $S$, is to make sense of the expressions $\tr f e^{-t S^2}$ and $\tr f |S|^{-z}$ as elements of $C_0(N)$. Sections~\ref{section:zeta-function},\ref{sec:summ-some-unbo} embark on the necessary theory. Open questions for further research remain, particularly in the direction of zeta residues and compatibility with the interior product in unbounded KK-theory.
Moreover, the case where the target algebra is noncommutative is still wide open.

\section{Preliminaries}
We start by recalling the notion of frames on Hilbert $C^*$-modules over $C^*$-algebras. For basic definitions on Hilbert $C^*$-modules, adjointable maps, tensor products, {\em et cetera} we refer to e.g. \cite{MR1222415}. We also recall the definition of unbounded Kasparov cycles \cite{MR715325}.

Keep in mind that we will, in later sections, specialize to the case of \emph{abelian} $C^*$-algebras, that is, those of the form $C_0(X)$ for $X$ a locally compact Hausdorff space.
Hilbert $C^*$-modules over such $C^*$-algebras are given by the sections of continuous fields of Hilbert spaces; cf. \cite{MR576842,MR163182}.

\subsection{Frames on Hilbert \cstar-modules}
We start this section by recalling two well-known results on Hilbert \cstar-modules. For completeness we include their (short) proofs.

\begin{proposition}
  Let $A$ be a \cstar-algebra and let $E_A$ be a Hilbert $A$-module.
  Then $E_A A$ is dense in $E$, and the map $u \colon v \otimes_A a \mapsto v a$, $E_A \otimes_A A \to E_A$, is unitary.
\end{proposition}
\begin{proof}
  Let $\{e_\lambda\}$ be an approximate unit of $A$. For all elements $v \in E_A$ one has $\ip{v - v e_\lambda, v - v e_\lambda} = \ip{v, v} - e_\lambda \ip{v, v} - \ip{v, v} e_\lambda + e_\lambda \ip{v, v} e_\lambda$, which converges to $0$; thus, $v$ is the norm limit of the sequence $v e_\lambda \in E_A A$.

  Clearly $u$ is isometric, so that its range is closed.
  As the range is dense, it must be surjective and, therefore, unitary.
%  As $u$ is an $A$-linear surjective isometric map, it must be unitary.
\end{proof}

\begin{proposition} \label{proposition:star-homomorphism-as-inner-product-of-modules}
  Let $A$ and $B$ be \cstar-algebras and let $E_A$ be a Hilbert $A$-module.
  If $\phi: A \to B$ is a $\ast$-homomorphism, then $B$ is a left $A$-module with the action $a \cdot b \eqdef \phi(a) b$.
  Then, there is an adjointable map
  \[
    \phi_* \colon E_A \to E_A \otimes_A B,
  \]
  such that $\ip{\phi_* v, \phi_* w}_B = \phi(\ip{v, w}_A)$.

  Moreover, if $T \in \lnmp{E_A}$ then $\phi_* T :=T \otimes 1$ is an adjointable endomorphism of $E_A \otimes_A B$, {\em i.e.} there is an induced map $\phi_* : \lnmp{E_A} \to \lnmp{E_A \otimes_A B}$. 
\end{proposition}
\begin{proof}
  Recall that the map $\id_* \colon v \cdot a \mapsto v \otimes_A a$, $E_A \to E_A \otimes_A A$, is an isomorphism.
  We set $\phi_* \eqdef (\id \otimes_A \phi) \circ \id_*$ and find that for $v \otimes_A a \in E_A \otimes_A A$ one has $\ip{\phi_* v, \phi_* v}_B = \ip{\phi(a), \ip{v, v} \cdot \phi(a)} = \phi(a)^* \phi(\ip{v, v}) \phi(a) = \phi(\ip{v \cdot a, v \cdot a})$.
\end{proof}

%\subsection{Frames on Hilbert \cstar-modules}

A convenient basic fact about separable Hilbert spaces $H$ is that they possess countable orthonormal bases $\{e_i\}$. For one thing, this allows one to explicitly relate the compact operators $B_0(H)$ to the direct limit $M(\C)$ of matrix algebras over the base field $\C$ and in particular to treat the trace on $L^1(H)$ using the series expression $\tr T = \sum_{i} \ip{e_i, T e_i}$.

The situation for Hilbert \cstar-modules is slightly less straightforward: we will introduce the analogous but strictly weaker concept of a \emph{frame}.
In spite of the increased generality, we will see that frames provide sufficient flexibility to mimic standard treatments of trace-class operators on Hilbert spaces in the setting of Hilbert $C_0(X)$-modules.

\begin{definition}
  Let $E_A$ be a countably generated Hilbert \cstar-module over a \cstar-algebra $A$.
  A \emph{frame} $\fram$ of $E_A$ is a sequence $\fram_i$ of elements of $E_A$, such that
  \begin{align*}
    \ip{v, w} &= \sum_{i=1}^\infty \ip{v, \fram_i} \ip{\fram_i, w},
  \end{align*}
  in norm, for all $v, w \in E_A$.
\end{definition}

Such objects $\fram$ were called `standard normalized frames' in \cite{MR1938798}.
Note that the subsequent treatment in \cite{MR1949886}, which is very similar to the definition used here, is different for non-unital \cstar-algebras: we require the $\fram_i$ to be in $E_A$, not in the `multiplier module' $\lnmp{A_A, E_A}$.
This choice will later imply, for instance, that we do \emph{not} consider the identity on the $C_0(X)$-module $C_0(X)$, for noncompact spaces $X$, to be in the trace class.

\begin{example} \label{example:p-of-frame-is-frame-hilbert}
  Let $H$ be a separable Hilbert space.
  Let $P \in B(H)$ be a projection and $K = PH \subset H$.
  Then, if $\{e_i\}$ is an orthonormal basis of $H$, we have
  \begin{align*}
    \ip{v, w} &= \sum_{i=1}^\infty \ip{Pv, e_i} \ip{e_i, Pw} = \sum_{i=1}^\infty \ip{v, P \fram_i} \ip{P \fram_i, w},
  \end{align*}
  for all $v, w \in K$. That is, $\fram = \{P e_i\}$ is a frame of $K$.
  Note that $\fram$ is not an orthonormal basis, because the $\fram_i$ might be neither orthogonal nor of norm 1.
\end{example}

Now, in the context of trace-class operators on a separable Hilbert space $H$, frames `are as good as orthonormal bases', in the sense of
Corollary~\ref{corollary:frame-trace-is-trace} below.

\begin{lemma} \label{lemma:l2-fram-framt-bound}
  Let $\fram, \framt$ be frames of a separable Hilbert space $H$ and let $T \in \lnmp{E_A}$.
  Then, the series $\sum_{i=1}^\infty \ip{T^* \framt_i, T^* \framt_i}$ converges if and only if $\sum_{i=1}^\infty \ip{T \fram_i, T \fram_i}$ converges, and the limits agree.
\end{lemma}
\begin{proof}
  Assume that $\sum_{i=1}^\infty \ip{T \fram_i, T \fram_i} < \infty$.
  Then for finite subsets $F \subset \N$,
  \begin{align*}
    \sum_{i \in F} \ip{T^* \framt_i, T^* \framt_i} &= \sum_{i \in F} \sum_{j =1}^\infty \ip{T^* \framt, \fram} \ip{\fram, T^* \framt}  \eqex \sum_{i \in F} \sum_{j=1}^\infty \ip{\fram, T^* \framt} \ip{T^* \framt, \fram} \\
    &\leq \sum_{j=1}^\infty \ip{T \fram, T \fram}.
  \end{align*}
  Being bounded and monotone, the series $\sum_{i = 1}^\infty  \ip{T^* \framt_i, T^* \framt_i}$ must converge.
  If we now switch $T$ and $T^*$, $\fram$ and $\framt$ and repeat the calculation, we see that the limits must in fact agree.
\end{proof}

\begin{corollary} \label{corollary:frame-trace-is-trace}
  Let $\fram$ be a frame of a separable Hilbert space $H$.
  Then, for $T \in \lnmp{H}$, $T \in \sch{2}{H}$ whenever $\sum_{i=1}^\infty \ip{T \fram_i, T \fram_i} < \infty$.
  Moreover, for $T \in \sch{1}{H}$, one has $\tr T = \sum_{i=1}^\infty \ip{\fram_i, T \fram_i}$.
\end{corollary}
\begin{proof}
  For the first part, let $\framt$ be an orthonormal basis and note that by Lemma~\ref{lemma:l2-fram-framt-bound}, $\sum_{i=1}^\infty \ip{T \fram_i, T \fram_i}$ converges whenever $\sum_{i=1}^\infty \ip{T^* \framt_i, T^* \framt_i}$ does, and the latter fact is equivalent to $T^* \in \sch{2}{H}$, which is equivalent to $T \in \sch{2}{H}$.

  For the second part, note that $ \sch{1}{H} =  \sch{2}{H} \sch{2}{H}$. It is thus sufficient to consider an element $T = |S|^2 \in \sch{1}{H}$ with $S  \in \sch{2}{H}$.  Then, $\tr |S|^2 = \tr |S^*|^2 = \sum_{i=1}^\infty \ip{S^* \framt_i, S^* \framt_i}$, which equals $\sum_{i=1}^\infty \ip{S \fram_i, S \fram_i}$ by Lemma~\ref{lemma:l2-fram-framt-bound}. 
\end{proof}

We will see later that the Example~\ref{example:p-of-frame-is-frame-hilbert} is a very good prototype for the general situation for Hilbert \cstar-modules as well.

\begin{example} \label{example:p-of-frame-is-frame-of-p-module}
Let $A$ be a unital \cstar-algebra and let $l^2(A) = l^2 \otimes_\C A_A$ be its standard module.
Let $\{e_i\}$ be the standard orthonormal basis of $l^2$ and define $\{\fram_i \eqdef e_i \otimes 1_A\}$ in $l^2(A)$. Then clearly
\begin{align*}
  \ip{v, w} &= \sum_{i=1}^\infty \ip{v, \fram_i} \ip{\fram_i, w}; \qquad (v,w \in l^2(A)).
\end{align*}
%for all $v, w \in l^2(A)$.% as norm limits in $A$. % and in $l^2(A)$, respectively.
If $P \in \lnmp{E_A}$ is a projection (i.e. $P^2 = P^* = P$) and $F_A = P \left( E_A \right)$, then
\begin{align*}
  %\ip{v, w} =
  \ip{Pv, Pw} &=\sum_{i=1}^\infty \ip{Pv, \fram_i} \ip{\fram_i,P w} = \sum_{i=1}^\infty \ip{v, P\fram_i} \ip{P\fram_i, w},
\end{align*}
so $\{P \fram_i\}$ is a frame of $F_A$.
\end{example}

Each frame of $E_A$ gives rise to a unitary $\theta_\fram \colon E_A \to (\theta_\fram \theta_\fram^*) l^2(A)$, as follows.

\begin{proposition}%[Frame transform]
  \label{proposition:frame-transform}
  %  Let $\fram$ be an element of $l^2 \otimes_\C E_A$.
  Let $\fram$ be a frame of $E_A$.
%  For $v \in E_A$, denote the map $E_A \to A$, $w \to \ip{w, v}$, by $\ket{v}$.
  %Then,
  The \emph{frame transform} $\theta_\fram \colon E_A \to l^2(A)$, given by
  \[
  \theta_\fram (v)% \eqdef (1 \otimes \ket{v})(\fram)
  := (\ip{\fram_i, v})_i
  \]
  is adjointable, and its adjoint satisfies $\theta_\fram^* (e_k\otimes a) = \fram_k \cdot a$.
  Moreover, $\fram$ is a frame if and only if $\theta_\fram^* \theta_\fram = \id_E$.
\end{proposition}
For the proof we refer to \cite[Theorem 3.5]{MR1949886}. 
%% \begin{proof}[{Proof (as in \cite{MR1949886})}]
%%   Note that, by Cauchy-Schwarz, $\norm{\ket{v}} = \norm{v}$ so that $\norm{\theta_\fram} \leq \norm{\fram}$.
%%   Now, on finite sequences in $l^2(A)$, the prospective $\theta_\fram^*$ is well-defined, and we have $\ip{\theta_\fram(v), ([i \leq n] a_i)_i} = \sum_{i \leq n} \ip{v, \fram_i} a_i = \ip{v, \sum_{i \leq n} \fram_i \cdot a_i} = \ip{v, \theta_\fram^*([i \leq n] a_i)}$.
%%   To see that $\theta_\fram^*$, which has dense domain, is bounded, let $x = ([i \leq n] a_i)_i$ in its domain and note that $\norm{\theta_\fram^* x} = \sup_{\norm{v} = 1} \norm{\ip{\theta_\fram^*x}} = \sup_{\norm{v} = 1} \norm{\ip{x, \theta_\fram (v)}} \leq \norm{\fram} \norm{x}$.

%%   For the statement about frames, note that $\ip{\theta_\fram(v), \theta_\fram(w)} = \ip{v, w}$ if and only if $\theta_\fram$ is a frame. If so, as $\theta_\fram$ is adjointable, we must have $\theta_\fram^* \theta_\fram = \id_E$.
%% \end{proof}

\begin{remark}
 Note that, unless $A$ is unital, the converse does not hold: not every isometry sending $E_A$ to a complemented submodule of $l^2(A)$ is induced by a frame.
  The frame elements $\fram_i$ would be given by $\theta_\fram^*(\delta_{ij} 1_A)_j$, but the latter is not an element of $l^2(A)$ unless $A$ is unital.
  This is where our treatment differs from that of \cite{MR1949886}, which works with frames (in the present sense) of the multiplier module $\lnmp{E_A, A}$ instead.
\end{remark}

Frames are compatible with $\ast$-homomorphisms.
In particular, this means that characters of a commutative \cstar-algebra map frames of Hilbert $C^*$-modules to frames of Hilbert spaces.

\begin{proposition} \label{proposition:phistar-of-frame-is-frame}
  Let $A$, $B$ be \cstar-algebras, $E_A$ a Hilbert $A$-module and $\phi: A \to B$ a $\ast$-homomorphism. If $\fram$ is a frame of $E_A$ and if $\phi$ is surjective, then $\phi_*(\fram)$ is a frame of $E_A \otimes_A B$.
\end{proposition}
\begin{proof}
  Consider $\framt \eqdef \phi_*(\fram) = \{\phi_*(\fram_i)\}_i \in E_A \otimes_A B$.
  Note that $\theta_\framt(\phi_* v) = \phi_* \theta_\fram(v)$, as elements of $\phi_* l^2(A) \subset l^2(B)$.
  Thus, $\ip{\theta_\framt(\phi_* v), \theta_\framt(\phi_* w)} = \phi(\ip{\theta_\fram(v), \theta_\fram(w)}) = \phi(\ip{v, w}) = \ip{\phi_* v, \phi_* w}$ so that with Proposition \ref{proposition:frame-transform} it follows that $\framt$ is a frame.
\end{proof}

\subsubsection{Existence of frames}
Kasparov's stabilization Theorem \cite{MR587371} shows that example \ref{example:p-of-frame-is-frame-of-p-module} describes the general unital (and, as we will see, the non-unital) case very well.

\begin{theorem}%[Kasparov's stabilization theorem \cite{MR587371}]
  \label{theorem:kasparov-stabilization}
  Let $A$ be a unital \cstar-algebra and let $E_A$ be a countably generated Hilbert $A$-module.
  Then there exists a projection $P^2 = P = P^*$ in $\lnmp{l^2(A)}$ such that $E_A \simeq P\left( l^2(A) \right)$.   In particular, $E_A$ possesses a frame.
\end{theorem}
For the proof we refer to \cite{MR587371} (see also \cite[Theorem 6.2]{MR1325694}).

The non-unital case requires more effort, but the end result is the same.
We refer to \cite[Section 2]{Kaad:DifferentiableAbsorptionHilbert} for a proof.

\begin{proposition}[{\cite[Proposition 2.6]{Kaad:DifferentiableAbsorptionHilbert}}] \label{proposition:existence-of-frames}
	Let $A$ be a \cstar-algebra.
	Then, all countably generated Hilbert $A$-modules possess a frame.
\end{proposition}

\subsection{The standard module over abelian \cstar-algebras}
\label{sect:stand-mod-abel}

For the rest of this subsection, let $X$ be a locally compact Hausdorff space.
The \cstar-algebra $C_0(X)$ is abelian -- and, by Gelfand duality, all abelian \cstar-algebras are of this type.

We will investigate the Hilbert $C_0(X)$-module $C_0(X, H)$, for $H$ a separable Hilbert space, which will later provide a useful tool in investigating the Schatten classes of operators on more general Hilbert $C_0(X)$-modules. We start with some basic definitions and results, whose proof we leave to the reader.

\begin{definition}
  Let $f$ be a map from a locally compact topological space $X$ to a normed space.
  We say that $f$ \emph{vanishes at infinity} whenever for all $\epsilon > 0$, $\{x \in X \colon \norm{f(x)} \geq \epsilon\}$ is compact.
\end{definition}

\begin{definition}
  Let $Y$ be a Banach space, equipped with its norm topology.
  The space $C_0(X, Y)$ consists of the continuous functions from $X$ to $Y$ that vanish at infinity.
\end{definition}

\begin{proposition} \label{proposition:c0-x-y-is-banach}
  Let $Y$ be a Banach space and $X$ be a locally compact topological space.
  Then, $C_0(X, Y)$ is a Banach space when equipped with the norm $\norm{f} \eqdef \sup_{x \in X} \norm{f(x)}$.
  Moreover, for $f \in C_0(X, Y)$, the map $x \mapsto \norm{f(x)}$ lies in $C_0(X)$.
\end{proposition}

\begin{proposition} \label{proposition:definition-of-c0-x-h}
  Let $H$ be a separable Hilbert space.
  Then, the Banach space $C_0(X, H)$ has the structure of a Hilbert \cstar-module when equipped with the $C_0(X)$-valued inner product $\ip{v, w}(x) \eqdef \ip{v(x), w(x)}_H$.
\end{proposition}
%% \begin{proof}
%%   The inner product is clearly right $C_0(X)$-linear and, by the polarization identity and the last part of Proposition~\ref{proposition:c0-x-y-is-banach}, lands in $C_0(X)$.
%%   Moreover, by the first part of Proposition~\ref{proposition:c0-x-y-is-banach}, the space $C_0(X,H)$ is complete with respect to the norm induced by the inner product.
%% \end{proof}

\begin{proposition}\label{proposition:c0-x-otimes-h-is-c0-x-h}
  The Hilbert $C^*$-module $C_0(X, H)$ is unitarily equivalent to the tensor product $H \otimes_\C C_0(X)$ of Hilbert $C^*$-modules.
\end{proposition}
\begin{proof}
  Let $\{e_i\}_i$ be an orthonormal basis of $H$, and for $v \in C_0(X, H)$ write $v_i \colon x \mapsto \ip{v(x), e_i}$; this defines a sequence of functions %% , as a composition of an element in $C_0(X, H)$ and one in $H^*$,  lies
  in $C_0(X)$.
  Because $\ip{v, v} = \sum_{i=1}^\infty v_i^* v_i$ converges pointwise, is positive and lies in $C_0(X)$, it converges in the norm of $C_0(X)$ by Dini's theorem.
  
  Consider then the $C_0(X)$-linear map $\theta \colon C_0(X, H) \to H \otimes_\C  C_0(X)$ defined by $v \mapsto \sum_i e_i \otimes v_i$. This series converges in $H \otimes_\C C_0(X) $ because
  $$
  \norm{\sum_{i \in F} e_i \otimes v_i} = \norm{\sum_{i,j \in F}  \ip{v_i, \ip{e_i, e_j}v_j}} = \norm{\sum_{i \in F} \ip{v_i, v_i}}.
  $$
  Moreover, taking limits on both sides we see that $\theta$ is isometric.

  Now, the map $m \colon H  \otimes_\C C_0(X) \to C_0(X, H), h \otimes a \mapsto ha$, is isometric because $\ip{\sum_i h_i \otimes a_i, \sum_i h_i \otimes a_i} = \sum_{ij} \ip{h_i, h_j}a_i^* a_j  = \ip{\sum_i m(h_i \otimes a_i), \sum_i m(h_i \otimes a_i)}$.
%  It must therefore extend to the Hilbert $C^*$-module completion.
  As $m$ inverts $\theta$ by orthonormality of the $\{e_i\}_i$, we conclude that $\theta$ is a surjective $C_0(X)$-linear isometry, so that it is adjointable (with adjoint $m$) and moreover unitary.
  Therefore, $C_0(X, H)$ is unitarily equivalent to $H \otimes_\C C_0(X)$.
\end{proof}

\begin{proposition}
  \label{prop:endo-free}
  The \cstar-algebra $\lnmp{C_0(X, H)}$ of adjointable endomorphisms of the Hilbert $C_0(X)$-module $C_0(X, H)$ is isomorphic to the \cstar-algebra $\cbstr{(X, B(H))}$ of bounded, \strongly{} continuous maps from $X$ to $B(H)$.
\end{proposition}
\begin{proof}
  Consider a point $x$ in $X$ as a character $\chi: C_0(X) \to \C$ given by evaluation in $x$. 
%  point $x \in X$ as a map (given by evaluation at $x$) from $C_0(X) \to \C$.
  By Proposition \ref{proposition:star-homomorphism-as-inner-product-of-modules} there is an adjointable map $\chi_*: C_0(X, H) \to C_0(X,H) \otimes_{C_0(X,H)} \C $ whose image $\chi_*(C_0(X, H))$ can be canonically identified with $H$. Accordingly, in terms of the map $\chi_* :  \lnmp{C_0(X, H)} \to \lnmp{C_0(X, H)}$ given by $T \mapsto T \otimes 1$, we have canonically $\chi_* \lnmp{C_0(X, H)} \simeq B(H)$. 

 Now take $T \in \lnmp{C_0(X, H)}$ and a convergent sequence $x_i \to x$ in $X$.
 Let $h \in H$ and let $v_0 \in C_0(X, H)$ be constant in a neighborhood of $x$ with $v_0(x) = h$.
% satisfy $v_0(y) = h$ for all $y$ in a neighbourhood of $x$.
Since $T$ is an adjointable endomorphism, both $Tv_0$ and $T^*v_0$ lie in $C_0(X, H)$, so that $\norm{Tv_0(x_i) - Tv_0(x)} \to 0$ and $\norm{T^*v_0(x_i) - Tv_0(x)} \to 0$. That is to say,  $\norm{({\chi_i}_* T)h - (\chi_*T)h} \to 0$ and $\norm{({\chi_i}_*T^*)h - (\chi_* T^*)h} \to 0$.
  As $x$ and $h$ were arbitrary, we conclude that $T \in C^{\mathrm{str}}{(X, B(H))}$.
  Moreover, if $\norm{T(x)} > C$ for some $x \in X$ and $C>0$, there is $v_0$ as above with $\norm{v_0} \leq 1$ and $\norm{Tv_0} > C - \epsilon$ for all $\epsilon > 0$, so that we conclude that $\norm{T} \geq C$.
  Thus, if $T$ preserves $C_0(X, H)$ it must lie in $\cbstr{(X, B(H))}$.

  Conversely, for all $T \in \cbstr{(X, B(H))}$ we have that $x \mapsto T(x)v(x)$ and $x \mapsto T^*(x)v(x)$ define maps in $C_0(X, H)$ for all $v \in C_0(X, H)$. Since $\ip{Tv, w}(x) = \ip{T(x)v(x), w(x)} = \ip{v(x), T^*(x)w(x)}$  we conclude that the pointwise adjoint provides an adjoint of $T$ as an operator of $C_0(X, H)$.
  That is, all such $T$ are adjointable operators on $C_0(X, H)$.
\end{proof}

\subsubsection{General Hilbert \cstar-modules over a commutative base}
We now apply the above results to the case of general Hilbert $C_0(X)$-modules.
For a deeper topological understanding of such modules, see \cite{MR576842,MR163182}.

\begin{proposition}\label{proposition:e-is-gamma0-x-p}
  Let $E$ be a Hilbert $C_0(X)$-module.
  Then there exists a \strongly{} continuous projection $P \in \cbstr(X, B(H))$ such that $E$ is isomorphic to the subset $\Gamma_0(X, P) \subset C_0(X, H)$ of those elements $h$ of $C_0(X, H)$ satisfying $h(x) \in \ran P_x$.
  Moreover, under this identification $\lnmp{E}$ is isomorphic to the set of elements $T \in \cbstr(X, B(H))$ such that $PT = TP = T$.
\end{proposition}
\begin{proof}
  If $X$ is compact, this is a direct consequence of Kasparov's stabilization Theorem~\ref{theorem:kasparov-stabilization} and Propositions~\ref{proposition:c0-x-otimes-h-is-c0-x-h} and \ref{prop:endo-free}.
  If not, consider $E$ as an $\unitization{C_0(X)}$-module and note that the endomorphism $P$ from Theorem~\ref{theorem:kasparov-stabilization} is already in $\cbstr(X^+, B(H))$ in terms of the one-point compactification $X^+$ of $X$. But since $(e, e) \in C_0(X)$ for $e \in E$, the map $P$ must project into a subspace of $C_0(X, H)$.

  For the last statement, let $T \in \cbstr(X, B(H))$ such that $PT = TP = T$.
  Then $PTP = T$ so that $T$ preserves $\Gamma_0(X, P)$.
  Conversely, let $S \in \lnmp{\Gamma_0(X, P)}$.
  Then, the map $T = SP$ is a composition of adjointable operators and is therefore an element of $\lnmp{C_0(X,H)} \simeq \cbstr(X, B(H))$ using Proposition \ref{prop:endo-free}. We clearly have $PT = TP = T$ so that the claim follows.
\end{proof}

Note that the fibers $P_x H$ of $E$ may vary quite wildly with $x \in X$, as the following example shows.

  \begin{example}
    Let $U \subset X$ be open and let $P$ be the orthogonal projection onto a closed subspace $V \subset H$. 
    Then, $E = C_0(U, V)$ is (in particular) a Hilbert $C_0(X)$-module because the action of $C_0(X)$ on $C_0(X, V)$ by pointwise multiplication preserves the subspace $C_0(U, V)$.
    The fibers of $E$ are $V$ for $x \in U$ and $\{0\}$, for $x \not \in U$.
  \end{example}

  The following example illustrates how the projections associated to such bundles behave.

  \begin{example}
    \label{example:locally-nontrivial-bundle-projection}
    Let $U \subset X$ be open and let $V = \operatorname{span}{v_0} \subset H$, with $\norm{v_0}_H = 1$.
    We will investigate the projection associated to the Hilbert $C_0(X)$-module $E= C_0(U, V)$ by Proposition~\ref{proposition:e-is-gamma0-x-p}.

    Let $\{\eta_i\}_i$ be a compactly supported partition of unity on $U$, so that $\sum_i \eta_i^2(x) = 1_U(x)$ for $x \in X$.
    Then, $v = \sum_i v_0 \eta_i^2 \ip{v, v_0}$ for all $v \in E$, so that $\{\fram_i \eqdef v_0 \eta_i \}_i$ is a frame of $C_0(U, V)$.
    In fact, any frame $\framt$ is of this form: we have $\eta_i = \ip{v_0, \framt_i}$.
    
    Now, let $\{w_i\}_i$ be an orthonormal basis of $H$.
    Then, $\theta_\fram$, the frame transform of $\fram$, maps $w \in C_0(U, V)$ to $\sum_i w_i \ip{\fram_i, w}$.

    Although the image $\Gamma_0(X, \theta_\fram \theta_\fram^*)$ of $\theta_\fram$ (consisting of those elements $w$ of $C_0(X, H)$ for which $\ip{w_i, w}$ has support contained in that of $\fram_i$) is isomorphic to $C_0(U, V)$ through the map $\theta_\fram^*$, it looks decidedly different from the isomorphic subspace $C_0(U, V)$ of $C_0(X, H)$ we started with.
    The associated projection $P = \theta_\fram \theta_\fram^*$ maps $w$ to $\sum_i w_i \eta_i \sum_j \eta_j \ip{w, w_j}$.
\end{example}

In the previous Example, note that $\norm{P(x)} = 1$ for $x \in U$ and $\norm{P(x)} = 0$ for $x \not \in U$.
Thus, $P \in \cbstr{(X, B(H))}$ lies in $C_b(X, B(H))$ if and only if $U$ is clopen, that is, if and only if the bundle $\{(x, h) \mid h = \indicator{U}(x) h\}$ is locally trivial.
This illustrates a general criterion for local triviality:

\begin{remark} \label{remark:norm-continuous-projections-make-locally-trivial-bundles}
  If $P \in C_b(X, B(H)) \subset \cbstr{(X, B(H))}$ is a projection, then each $x \in X$ has a neighbourhood on which $\norm{P(y) - P(x)} < 1$, so that there exists a continuous map $y \mapsto u_y$ with $P(y) = u_y P(x) u_y^*$ by \cite[Proposition 5.2.6]{MR1222415}.
  We conclude that the bundle $\{(x, h) \mid h \in P(x) H\}$ is locally trivial.
\end{remark}

Conversely, if the bundle is locally trivial, at least when the fibers are constant, we can choose $P$ to be norm continuous. More precisely, 
%\begin{remark}
%L
if we let $p: F \to X$ be a locally trivial bundle of Hilbert spaces with separable, infinite-dimensional fiber $H$.
%  If $H$ is infinite-dimensional,
By \cite[Corollary 4.79]{MR1634408} $F$ is isomorphic to $X \times H$, and $\Gamma_0(F)$ is isomorphic to $C_0(X, H)$. We may thus choose $P = \id_H \in C_b(X, B(H))$ in Proposition~\ref{proposition:e-is-gamma0-x-p}.

  If instead $H$ is finite-dimensional, then by the Serre-Swan theorem \cite[Theorem 2.10]{MR1789831} there exists a projection $p \in M_n(A)$ with $pA^n \simeq \Gamma_0(F)$.
  Thus, in Proposition~\ref{proposition:e-is-gamma0-x-p} we may choose $H = \C^n$ and $P = p$.
%\end{remark}

\section{Schatten classes for Hilbert $C_0(X)$-modules}

\label{chapter:abelian-schatten-classes}

When $A$ is abelian, {\em i.e.} $A \simeq C_0(X)$ for some locally compact Hausdorff space $X$, each adjointable operator $T$ on a Hilbert $A$-module $E_A$ can be localized by the pure states $x$ of $A$ to yield a family $\chi_* T$ of operators on the Hilbert spaces $\chi_* E_A$. We will unify the pointwise Schatten classes $\sch{p}{\chi_* E_A}$ into a two-sided ideal $\sch{p}{E_A} \subset \lnmp{E_A}$ and define an $A$-valued trace on $\sch{1}{E_A}$.

% The spectrum\footnote{With multiplicities.} is a complete unitary invariant of compact normal operators on Hilbert spaces. It is an understatement to say that many interesting theorems, especially from the perspective of noncommutative geometry, are essentially results about the spectra of certain operators.

% The spectrum of a trace-class operator uniquely determines, by Lidskii's theorem, its trace. Indeed, it is natural for unitary (and thus, spectral) invariants to be expressed through the trace.
% This is not an accident, as can be illustrated as follows: by the identity theorem for Dirichlet series, the (holomorphic) function $z \mapsto \tr T^z$ on the half-plane $\Re z > 1$ uniquely determines the spectrum of $0 \leq T \in \sch{1}{H}$; cf. \cite{MR0185094} and Section~\ref{subsection:zeta-function} below.

% Now, for (compact) operators on Hilbert \cstar-modules, the usual notion of spectrum is far less descriptive than for compact operators on Hilbert spaces; the spectrum of an endomorphism $f$ of the Hilbert $C_0(X)$-module $C_0(X)$, for instance, only determines its essential range. It is therefore important to work with a more refined invariant.
% For Hilbert $C^*$-modules over abelian \cstar-algebras, one does of course have access to the \emph{pointwise} spectrum of compact operators, but the relation to the operator itself is complicated by the usual discontinuity of the spectral projections (see e.g. \cite[IV.3]{MR1335452}).

\begin{assumption}
  We will require all of our Hilbert $A$-modules to be countably generated in order to ensure access to frames using Proposition~\ref{proposition:existence-of-frames}.
%Thus, for the remainder of this paper, let $A$ be a separable, abelian \cstar-algebra and let $E_A$ be a countably generated Hilbert $A$-module.

We will denote the character space of a commutative $C^*$-algebra $A$ equipped with the weak$^*$ topology by $\gfdual{A}$, so that $A \simeq C_0(\gfdual{A})$ with $\gfdual{A}$ a locally compact Hausdorff space.
\end{assumption}

The {\em least} we should demand of `Schatten-class operators' $T$ on $E_A$ is that their pointwise Schatten norm, {\em i.e.} the trace of $|x_ *T|^p$, varies continuously with $x \in \gfdual{A}$. In fact, this is the way to ensure that the `trace-class operators' have traces with values in $A$ and that the other Schatten classes respect this property in their pairing.
The \emph{most} we could reasonably demand, in contrast, is that the operators $\chi_* T$ are continuous in Schatten norm with respect to some trivialization of $E_A$ (see Definition~\ref{definition:continuous-schatten-class}, below).
It will turn out that these requirements, properly understood, are equivalent and yield a well-behaved Schatten class.

\begin{definition} \label{definition:minimal-definition-of-schatten-class}
  The {\em $p$-th Schatten class} $\sch{p}{E_A}$ for $1 \leq p < \infty$ is the space of all endomorphisms $T \in \lnmp{E_A}$ for which the function $\tr |T|^p \colon \gfdual{A} \to \R \cup \{\infty\}, \chi \mapsto \tr |\chi_* T|^p$ %\norm{\chi_* T}_p$,
  lies in $A$.
\end{definition}

The following proposition, familiar from the Hilbert space case, is immediate from the definition:

\begin{proposition} \label{proposition:t-in-lp-iff-tabs-in-lp-iff-tp-in-l1}
  Let $1 \leq p < \infty$ and let $T \in \lnmp{E_A}$. Then
  $$
T \in \sch{p}{E_A} \iff |T| \in \sch{p}{E_A} \iff |T|^{p} \in \sch{1}{E_A}.
  $$
  %% The following are equivalent:
  %% \begin{itemize}
  %% \item $T \in \sch{p}{E_A}$.
  %% \item $|T| \in \sch{p}{E_A}$.
  %% \item $|T|^{p} \in \sch{1}{E_A}$.
  %% \end{itemize}
\end{proposition}

\begin{remark}
  Recall that Dini's theorem, translated to the abelian \cstar-algebraic context, states the following: if $a_i$ is a sequence of positive elements in $A$, then $\sum_i a_i$ converges in norm if and only if the function $x \mapsto \sum_i x(a_i)$ is an element of $C_0(X) \simeq A$.
  This theorem plays a crucial role throughout, because it allows us to relate the fiberwise Schatten norms on bundles of Hilbert spaces to various expressions for the element $\tr |T|^p \in A$, for $T \in \sch{p}{E_A}$.
\end{remark}

We will use the existence of frames ({\em cf.} Proposition~\ref{proposition:existence-of-frames}), to relate $\sch{p}{E_A}$ to the Schatten classes $\sch{p}{l^2(A)}$ on the standard module $l^2(A)$ and to relate the trace $\tr |T|^p$ to a series expression in terms of (arbitrary) frames.

\begin{theorem} \label{theorem:nonstandard-schatten-class-is-pullback-by-frame}
  Let $T \in \lnmp{E_A}$.
  Then $T \in \sch{p}{E_A}$ if and only if $\theta_\fram T \theta_\fram^* \in \sch{p}{l^2(A)}$ for any frame $\fram$ of $E_A$ with frame transform $\theta_\fram$.
  Equivalently, $T \in \sch{p}{E_A}$ if and only if the series $\sum_{i=1}^\infty \ip{\fram_i, |T|^p \fram_i}$ converges in norm; the limit equals $\tr |T|^p = \tr \theta_\fram |T|^p \theta_\fram^*$.
\end{theorem}
\begin{proof}
  We start with the second part.
  As $\chi_*(\fram)$ is a frame of the Hilbert space $\chi_* E_A$, one has $\tr |\chi_*T|^p = \sum_{i=1}^\infty \ip{\chi_*(\fram_i), |\chi_*T|^p \chi_*(\fram_i)} = \sum_{i=1}^\infty x(\ip{\fram_i, |T|^p \fram_i})$. Hence if $T \in \sch{p}{E_A}$, the positive series $\sum_{i=1}^\infty \ip{\fram_i, |T|^p \fram_i}$ converges in norm to an element $\tr |T|^p$ of $A$ by Dini's theorem.
  
  Conversely, if the series $\sum_{i=1}^\infty \ip{\fram_i, |T|^p \fram_i}$ converges in norm, then the limit provides an element $\tr |T|^p \in A$ such that $\chi(\tr |T|^p) = \tr |\chi_* T|^p$ for all characters $\chi$ of $A$, so that $T \in \sch{p}{E_A}$.

  For the first part, let $\{e_i\}_i$ be the standard orthonormal basis of $l^2$, so that $\theta_\fram^*(e_i \otimes 1) = \fram_i$.
  Because $\theta_\fram^* \theta_\fram = \id_{E_A}$, we have $|\theta_\fram T \theta_\fram^*|^p = \theta_\fram |T|^p \theta_\fram^*$.  Then, for any $\chi \in \gfdual{A}$ we have
  $$
  \tr \chi_* \theta_\fram |T|^p \theta_\fram^* = \sum_{i=1}^\infty \ip{e_i, \chi_* \theta_\fram |T|^p \theta_\fram^* e_i} = \sum_{i=1}^\infty  \chi\left(\ip{\fram_i, |T|^p \fram_i}\right).
  $$
  Thus, if $\sum_{i=1}^\infty \ip{\fram_i, |T|^p \fram_i}$ converges to an element of $A$, we have $\theta_\fram |T|^p \theta_\fram^* \in \sch{p}{l^2(A)}$. Conversely, if $\theta_\fram |T|^p \theta_\fram^* \in \sch{p}{l^2(A)}$ then the function on $\gfdual{A}$ defined by $\chi \mapsto \sum_{i=1}^\infty \chi\left(\ip{\fram_i, |T|^p \fram_i}\right)$ lies in $C_0(X) \simeq A$.
  The series must then converge in norm by Dini's theorem.
\end{proof}

\begin{corollary} \label{corolllary:s-in-Lp-if-sp-leq-tp-in-Lp}
  Let $S \in \lnmp{E_A}$ and $T \in \sch{p}{E_A}$.
  If $|S|^p \leq |T|^p$, then $S \in \sch{p}{E_A}$ and in particular $\tr |S|^p \leq \tr |T|^p$.
\end{corollary}
\begin{proof}
  Let $\fram$ be a frame of $E_A$.
  Then $\sum_{i \in F} \ip{\fram_i, |S|^p \fram_i} \leq \sum_{i \in F} \ip{\fram_i, |T|^p \fram_i}$ for all finite $F \subset \N$; in particular, the left-hand side is Cauchy whenever the right-hand side is.
  By Theorem~\ref{theorem:nonstandard-schatten-class-is-pullback-by-frame}, this will suffice.
\end{proof}

\begin{remark}
  The above Corollary is weaker than the Hilbert space version ({\em cf.} Lemma~\ref{lemma:schatten-norm-monotone}(\ref{lemma:schatten-norm-monotone:item:3}) below.
 Instead, Corollary~\ref{corollary:lp-order-ideal-in-compacts} below gives a stronger result but an additional assumption on $S$ is required.
  Note that this is the only point in the treatment of this chapter where such a difference between the Hilbert module and Hilbert space Schatten classes appears.
\end{remark}

The most straightforward road to analyzing the structure and properties of $\sch{p}{E_A}$ now lies open: we will investigate $\sch{p}{l^2(A)}$ and use the pullback by the frame transforms to transfer its properties to $\sch{p}{E_A}$.
It will turn out that $\sch{p}{l^2(A)}$ is indeed very well-behaved, so that this allows us to recover many of the familiar properties of the Schatten classes of operators on Hilbert spaces.

\subsection{The Schatten class on the standard module}
\label{section:schatten-class-standard-module}
Let $H$ be a separable Hilbert space and let $A=C_0(X)$ be an abelian $C^*$-algebra. Recall from Section \ref{sect:stand-mod-abel} that the Hilbert $A$-module $H \otimes_\C A$ is isomorphic to $C_0(X,H)$ through the canonical isomorphism $\chi_*(H \otimes_\C A) \simeq H$ (with a character $\chi\in \widehat{A}$ always corresponding via the Gelfand transform to a point $x \in X$). Recall, moreover, that its endomorphism space is given by $\cbstr(X, H)$ with the fiberwise action given simply by $T(h)(x) \eqdef T(x) (h(x))$ for $x \in X$.

Because all fibers are identified canonically with $H$, we may canonically compare the localizations of an operator $T \in \lnmp{H \otimes_\C A}$ between \emph{different} fibers.
This technically useful difference between $H \otimes_\C A$ and other Hilbert $A$-modules will allow us to use topologies on $B(H)$ to define particular subsets of $\lnmp{H \otimes_\C A}$.
Most importantly,

\begin{definition} \label{definition:continuous-schatten-class}
  The space of {\em continuous Schatten-class operators} on $C_0(X,H) \simeq H \otimes_\C A$ is the subspace $C_0(X, \sch{p}{H})$ of $\cbstr(X, B(H))$.
\end{definition}

Note that where the requirement that $\tr |T|^p \in A$ is the least restrictive
among reasonable criteria for a `Schatten-class operator' $T$, as
discussed above Definition~\ref{definition:minimal-definition-of-schatten-class},
the condition of Definition~\ref{definition:continuous-schatten-class}
is arguably the \emph{most} restrictive.

However, we will prove that the demands are, in fact, equivalent, so that the space $\sch{p}{H \otimes_\C A}$ of Schatten-class operators can be identified with the Banach space $C_0(X, \sch{p}{H})$ of continuous Schatten-class operators. 
This will later ---specifically in Theorem ~\ref{theorem:schp-is-ideal-and-norm-has-properties}---allow us to combine the properties that follow straightforwardly from either of the two definitions.

\begin{remark} \label{remark:topological-schatten-class-in-operator-schatten-class}
  Clearly, one has $C_0(X, \sch{p}{H}) \subset \sch{p}{C_0(X,H)}$ because, for $T \in C_0(X, \sch{p}{H})$ and $x, y \in X$, we have $\norm{T(x) - T(y)}_p \geq | \norm{T(x)}_p - \norm{T(y)}_p |$ so that $x \mapsto \norm{T(x)}_p \in A$. 
\end{remark}

%% \begin{remark}
%%   As a space of continuous Banach space-valued functions, $C_0(X, \sch{p}{H})$ is itself a Banach space with the norm $\norm{T}_p \eqdef \sup_{x \in \gfdual{A}} \norm{T(x)}_p = \norm{\tr |T|^p}_{A}$.
%% \end{remark}

The ostensibly more restrictive definition of the continuous Schatten class has some advantages to that of $\sch{p}{C_0(X,H)}$.
For instance, it is immediate from the definition that $C_0(X, \sch{p}{H})$ is closed under addition, as we have \emph{not} yet shown for $\sch{p}{C_0(X,H)}$.
Moreover, we can easily obtain a continuous version of the H\"older--von Neumann inequality:

\begin{proposition} \label{proposition:continuous-hoelder-von-neumann-inequality}
 Let $p, q, r \geq 1$ such that $\frac1{p} + \frac1{q} = \frac1{r}$, and let $S \in C_0(X, \sch{p}{H})$ and $T \in C_0(X, \sch{q}{H})$. Then $ST \in C_0(X, \sch{r}{H})$ and $\norm{ST}_r \leq \norm{S}_p \norm{T}_q$.
\end{proposition}
\begin{proof}
  For any $x, y \in X$ we have
  \[
    \norm{ST(x) - ST(y)}_r \leq \norm{(S(x) - S(y)) T(x) + S(y) (T(x) - T(y))}_r,
  \]
  which is bounded by $\norm{S(x) - S(y)}_p \norm{T(x)}_q + \norm{S(y)}_p \norm{T(x) - T(y)}_q$ due to the H\"older--von Neumann inequality \cite[Theorem 2.8]{MR2154153}.
  We conclude that $ST$ is continuous as a map from $X$ to $\sch{r}{H}$. Moreover, since $\norm{ST(x)}_r \leq \norm{S(x)}_p \norm{T(x)}_q$ for all $x$, the statement on the norms follows as well as the claim that $ST \in C_0(X, \sch{r}{H})$.
\end{proof}

We now identify a subset of $C_0(X, \sch{p}{H}) \subseteq \sch{p}{H \otimes_\C A}$ ({\em cf.} Remark ~\ref{remark:topological-schatten-class-in-operator-schatten-class} for the latter inclusion) whose completion in the Banach norm of $C_0(X, \sch{p}{H})$ is all of $\sch{p}{H \otimes_\C A}$. This will show that $\sch{p}{H \otimes_\C A}$ coincides with the Banach space $C_0(X, \sch{p}{H})$. Moreover, the fact that this common subset consists of finite-rank operators, in the Hilbert module sense, allows us to show that $\sch{p}{E_A} \subset \cptmp{E_A}$ in Theorem~\ref{theorem:schp-is-ideal-and-norm-has-properties}.%Corollary~\ref{corollary:c0x-schp-compact}.

\begin{proposition}
The finite-rank operators (in the Hilbert \cstar-module sense) on $C_0(X, H)$ lie in $C_0(X, \sch{p}{H})$.
\end{proposition}
\begin{proof}
Let $T = \ket{v}\bra{w}$ with $v,w \in C_0(X,H)$. 
Then, for $x, y \in X$, $\norm{(T_x - T_y)}_p \leq \norm{ \ket{v_x - v_y} \bra{w_y} }_p + \norm{\ket{v_y} \bra{w_x - w_y} }_p$.
Pointwise in $H$, however, we have $\norm{ \ket{\xi} \bra{\eta}}_p = \norm{\xi} \norm{\eta}$, so that norm continuity of $v$ and $w$ finish the proof.
\end{proof}

\begin{lemma} \label{lemma:trivial-uniformly-finite-rank-so-schatten}
  Let $V \subset H$ be finite-dimensional and consider $C_0{(X, B(V))}$ as a subspace of $\cbstr{(X, B(H))}$ by the map $T \mapsto T \oplus 0 \in B(V) \oplus B(V^{\perp}) \subset B(H)$.
  Then, all elements of $C_0(X, B(V))$ are finite rank operators on the Hilbert modules $C_0(X, H)$.
In particular, we have $C_0{(X, B(V))} \subset C_0(X, \sch{p}{H})% \cap \sch{p}{C_0(X,{H})}
  $ for all $1 \leq p < \infty$.
\end{lemma}
\begin{proof}
  Let $T \in C_0(X, B(V))$ and decompose so that $T = S |T|^{\frac12}$.
  Then let $\{e_i\}_{i=1}^n$ be an orthonormal basis of $V$.
  Because $\sum_{i} \ket{e_i} \bra{e_i} = \id_V$, we have $T = S \sum_{i} \ket{e_i} \bra{e_i} |T|^{\frac12} = \sum_i \ket{S e_i} \bra{|T|^{\frac12} e_i}$, where we denote for $R \in C_0(X, B(V))$ by $R e_i$ the element $x \mapsto R(x) e_i$ of $C_0(X, V)$. We conclude that $T$ is of finite rank.
\end{proof}

\begin{remark}
By a theorem of Fell \cite[Theorem 4.1]{MR164248}, the compact operators on $C_0(X, H)$ that have bounded rank automatically have continuous trace.
See remark~\ref{remark:comparison-to-dixmier-continuous-trace}, below, for the link to the study of continuous-trace \cstar-algebras.
\end{remark}

\subsubsection{Some properties of the Schatten classes on Hilbert spaces}

We assemble here some more or less well-known properties of the ordinary Schatten classes on $B(H)$.
The purpose is to show that one can use the series that defines the trace of $|T|^p$ to control the rate at which certain finite-rank approximations of $T$ will converge to $T$ in Schatten norm.
\begin{lemma} \label{lemma:schatten-norm-monotone}
  Let $T \in \sch{p}{H}$.
  \begin{enumerate}
      \item  For $1\leq  p <2 $ one has $\norm{T}_p^p = \inf_{\{e_i\}_i} \sum_{i=1}^\infty \norm{T e_i}^p$, where the infimum is taken over orthonormal bases $\{e_i\}_i$ of $H$.
  \item For $p \geq 2$, one has $\norm{T}_p^p = \sup_{\{e_i\}_i} \sum_{i=1}^\infty \norm{T e_i}^{p}$, where the supremum is over orthonormal bases $\{e_i\}_i$ of $H$.

  \item \label{lemma:schatten-norm-monotone:item:3} Let $p \geq 2$. For $S \in \lnmp{H}$, if $|S|^2 \leq T$ for some $T \in \sch{p/2}{H}$, then $S \in \sch{p}{H}$ and $\norm{S}_{p}^p \leq \norm{T}_{p/2}^{p/2}$.
  \end{enumerate}
\end{lemma}
\begin{proof}
  Let $\{e_i\}_i$ be an eigenbasis of (the compact, normal operator) $T^*T$, ordered by decreasing of the corresponding eigenvalues $\{\lambda_i\}$. %; the $i$'th eigenvalue $\lambda_i$ of $T^*T$ equals $\ip{T e_i, T e_i}$.
  
  First note that $\tr |T|^p = \sum_{i=1}^\infty \ip{T e_i, T e_i}^{p/2}$.
  Any other orthonormal basis $\{f_i =U e_i\} $ of $H$ is related to $\{ e_i\}$ by some unitary operator $U \in B(H)$.

  For $1 \leq p < 2$, the function $x \mapsto x^{p/2}$ is concave on $\R_+$. Thus, since $\ip{T f_i, T f_i} = \sum_j \ip{f_i, T^* T e_j} \ip{e_j, f_i} = \sum_j \lambda_j \left| \ip{e_j, f_i} \right|^2$ we find that $\norm{T f_i}^{p} \geq \sum_j \lambda_j^{p/2} \left|\ip{e_j, f_i} \right|^2 = \ip{f_i, |T|^p f_i}$. We conclude that
  $$\sum_{i=1}^\infty \norm{T f_i}^{p} \geq \sum_{i=1}^\infty \ip{f_i, |T|^p f_i} = \tr U^* |T|^p U = \tr |T|^p.$$
  
  For $p \geq 2$ the function $x \mapsto x^{p/2}$ is convex on $\R_+$ and we find, {\em  mutatis mutandis} in the argument as above that now 
$$\sum_{i=1}^\infty \norm{T f_i}^{p} \leq \sum_{i=1}^\infty \ip{f_i, |T|^p f_i} = \tr U^* |T|^p U = \tr |T|^p.$$
  
  %% Thus, since $\ip{T f_i, T f_i} = \sum_j \ip{f_i, T^* T e_j} \ip{e_j, f_i} = \sum_j \lambda_j \left| \ip{e_j, f_i} \right|^2$ we find $\norm{T f_i}^{p} \leq \sum_j \lambda_j^{p/2} \left|\ip{e_j, f_i} \right|^2 = \ip{f_i, |T|^p f_i}$.
  %% We conclude that $$\sum_{i=1}^\infty \norm{T f_i}^{p} \leq \sum_{i=1}^\infty \ip{f_i, |T|^p f_i} = \tr U^* |T|^p U = \tr |T|^p.$$
  %%    For $1 \leq p < 2$, note that $x \mapsto x^{p/2}$ is concave on $\R_+$, so that, with mutatis mutandis the same argument as above, $\norm{T f_i}^{p} \geq \ip{f_i, |T|^p f_i}$; all inequalities are thereby reversed.

  For the final claim, if $|S|^2 \leq R^* R$, one has $\norm{Se_i}^p \leq \norm{R e_i}^p$ so that $\norm{S}_p^p \leq \norm{R}_p^p = \norm{R^* R}_{p/2}^{p/2}$.
\end{proof}

We will need the following Corollary in the proof of Lemma~\ref{lemma:lph-uniform-completion-of-matc}.

\begin{corollary} \label{corollary:schatten-p-norm-of-truncated-operator-using-truncated-trace}
Let $T \in B(H)$ and let $e$ be a finite-rank projection. For any $p \geq 2$ we have $Te \in \sch{p}{H}$ and, in fact,
\[
\norm{Te}_p^p  \leq \tr e |T|^p e.
\]
\end{corollary}
\begin{proof}
As in Lemma~\ref{lemma:schatten-norm-monotone}, let $T^*T$ have eigenbasis $\{g_i\}$ with eigenvalues $\{\lambda_i\}$.
Then, for any $v \in H$, we have $\ip{T e v, T e v} = \ip{|T|^2 e v, e v} = \sum_j \lambda_j |\ip{ev, g_j}|^2$. % by Parseval's identity and polarization.
In particular,
$$
\ip{T ev , T ev}^{p/2} =\norm{e v}^{p} \left( \sum_j  \lambda_j \frac{| \ip{ev, g_j} |^2}{\ip{ev, ev}} \right)^{p/2}.
$$
By convexity of $x \mapsto x^{p/2}$ on $\R_+$ for $p \geq 2$ we have
$$
\ip{T e v, T ev}^{p/2} \leq \norm{ev}^{p - 2} \sum_j \lambda_j^{p/2} | \ip{ev, g_j} |^2 = \norm{ev}^{p-2} \ip{ev, |T|^p ev}.
$$
But from the Lemma it then follows that
\[
\norm{Te}_p^p = \sup_{\{f_i\}} \sum_i \norm{Te f_i}^p \leq \sup_{\{f_i\}} \sum_i \norm{e f_i}^{p-2} \ip{ef_i, |T|^p ef_i} \leq \tr e |T|^p e.\qedhere
\]
\end{proof}

With respect to a choice of orthonormal basis on $H$, we can view $\sch{p}{H}$ as a completion of the direct limit $\mat \C$ of finite matrix algebras $\mat_n(\C)$ in the Schatten norm.
This can be done `uniformly', where the convergence of the limit is controlled by the trace, as we will show in Proposition~\ref{proposition:lph-uniform-completion-of-matc} below.

Given $T \in \lnmp{H}$ and a sequence of increasing finite-dimensional subspaces $PH \subset H$, the operator $P T P$ converges to $T$ \strongly{} as $P \to \id_H$.
The following Lemma allows us to control the $p$-norm of the difference when increasing the rank of $P$ by one, such that $P T P \to T$ in $p$-norm precisely when $T \in \sch{p}{H}$.

\begin{lemma} \label{lemma:lph-uniform-completion-of-matc}
  Let $p \in [1, \infty)$ and let $T \in \sch{p}{H}$.
  Let $e$ be a finite-rank projection in $H$, let $P$ be a finite-rank projection with $Pe = eP = 0$ and let $Q = P + e$.
  Then,
  \[
	\norm{QTQ - PTP}_p \leq \|T\|_p^{1/2} \left( \left( \tr e |T^*|^p e \right)^{1/2p} + \left( \tr e |T|^p e \right)^{1/2p} \right)
  \]
\end{lemma}
\begin{proof}
	One has $QTQ - PTP = PTe + eTe + eTP = PTQ + eTP$, so that $\| QTQ - PTP \|_p \leq \|eTQ\|_p + \|PTe\|_p$ by cyclicity of the Schatten norms.
	Now compose $T$ as $T = S |T|^{\frac12}$ with $|S|^2 = |T|$.
	Then, by the Hölder-von Neumann inequality, we have $\|eTQ\|_p \leq \|eS\|_{2p} \||T|^{1/2} Q\|_{2p}$ and similarly $\|PTe\|_p \leq \|PS\|_{2p} \||T|^{1/2} e\|_{2p}$.
	Moreover, $\|PS\|_{2p} \leq \|T\|_p^{1/2}$ and $\||T|^{1/2} Q\|_{2p} \leq \|T\|_{p}^{1/2}$, and as $\|eS\|_{2p} = \|S^* e\|_{2p}$, we may now apply Corollary~\ref{corollary:schatten-p-norm-of-truncated-operator-using-truncated-trace} directly (to $\|S^* e\|_{2p}$ and $\||T|^{1/2} e\|_{2p}$) in order to finish the proof.
\end{proof}

\subsubsection{Identification of $\sch{p}{H \otimes_\C A}$ with $C_0(X, \sch{p}{H})$}

%\begin{remark} \label{remark:diagonalized-approximations-do-not-work-on-c0x}
It is a well-known fact that the finite-rank operators are dense in
$\sch{p}{H}$ for any $p$, as can easily be seen from the spectral theorem for compact self-adjoint operators. % matrices and the decomposition $2T = (T + T^*) + i(T - T^*)/i$.
This argument, however, does not extend uniformly to self-adjoint $T \in C_0(X, \sch{p}{H})$ unless $T$ is continuously diagonalizable.
%\end{remark}
The explicit, albeit apparently somewhat clumsy, result of
Lemma~\ref{lemma:lph-uniform-completion-of-matc} presents a solution to this problem. % situation of Remark~\ref{remark:diagonalized-approximations-do-not-work-on-c0x}.
Namely, orthogonal projections $P \in B(H)$ can be lifted to constant projections $\cbstr{(X, B(H))}$. %, are independent of $x \in \gfdual{A}$.
This allows the Lemma to be applied uniformly to all of $\cbstr{(X, B(H))}$, as we do in Proposition~\ref{proposition:lph-uniform-completion-of-matc} below.
This will then provide the main ingredient of main result of this section, to wit, the identification of $\sch{p}{H \otimes_\C A}$ with $C_0(X, \sch{p}{H})$ (Theorem~\ref{theorem:lp-iff-tr-tp-in-c0x}).

\begin{proposition} \label{proposition:lph-uniform-completion-of-matc}
  Let $T \in \cbstr{(X, B(H))}$, let $p \in [1, \infty)$ and assume that the function $\tr |T|^p \colon x \in X\mapsto \tr |T(x)|^p$ is defined everywhere and lies in $A= C_0(X)$.
  
  Let $\{e_i\}_i$ be an orthonormal basis of $H$ and let, for any $n \geq 0$, $P_n \eqdef \sum_{i=1}^n \ket{e_i} \bra{e_i}$ be the corresponding spectral projections.
  Then the operators $T_n \eqdef P_n T P_n$ (as in Lemma~\ref{lemma:lph-uniform-completion-of-matc}) are elements in $C_0(X, \sch{p}{H})$ and, for $m \geq n$,
\begin{align*}
  \norm{T_m - T_n}_p^{2p} &\leq 2^{2p-1} \sup_{x \in X} \tr |T(x)|^p \sum_{i=n+1}^m \ip{e_i, |T(x)|^p + |T^*(x)|^p e_i}
\end{align*}
\end{proposition}
\begin{proof}
	Note that $T_n$ is a finite-rank operator and in particular $T_n \in C_0(X, \sch{p}{H})$, by Lemma~\ref{lemma:trivial-uniformly-finite-rank-so-schatten}.
% (x/2 + y/2)^{2p} \leq 1/2 x^{2p} + 1/2 y^{2p}
% so (x + y)^{2p} * 2^{-2p} \leq 1/2 (x^{2p} + y^{2p})
% so (x+y)^{2p} \leq 2^{2p-1} (x^{2p} + y^{2p})
	Moreover, by Lemma~\ref{lemma:lph-uniform-completion-of-matc} applied to the projections $P = P_n$, $Q = P_m$ and $e = P_m - P_n$, and by convexity of $x \mapsto x^{2p}$, we have $\norm{T_m - T_n}_p^{2p} \leq 2^{2p-1} \|T\|_{p}^{p} \sum_{i=n+1}^{m} \ip{e_i, |T  |^p + |T^*|^p e_i}$.
%   Note that $P_n T P_n \in \cbstr(X, B(P_n H))$. Since the norm and strong topologies on $B(P_n H)$ coincide by finite-dimensionality, $T_n$ is norm continuous and as $\norm{T_n(x)}^p \leq \tr |T(x)|^p$ we have $T_n \in C_0(X, B(P_n H))$.
%   Thus, by Lemma~\ref{lemma:trivial-uniformly-finite-rank-so-schatten} we have $T_n \in C_0(X, \sch{p}{H})$.
% Then, the estimate on $\norm{T_m - T_n}_p^{2p}$ is a verbatim application of Lemma~\ref{lemma:lph-uniform-completion-of-matc} to their pointwise Schatten norms, together with the simple estimate $\norm{P_m |T(x)|^p P_m}_1 \leq \norm{|T(x)|^p}_1$.
\end{proof}

% \begin{proposition} \label{proposition:lph-uniform-completion-of-matc}
%   Let $T \in \cbstr{(X, B(H))}$, let $p \in [1, \infty)$ and assume that the function $\tr |T|^p \colon x \in X\mapsto \tr |T(x)|^p$ is defined everywhere and lies in $A= C_0(X)$.
  
%   Let $\{e_i\}_i$ be an orthonormal basis of $H$ and let, for any $n \geq 0$, $P_n \eqdef \sum_{i=1}^n \ket{e_i} \bra{e_i}$ be the corresponding spectral projections.
%   Then the operators $T_n \eqdef P_n S P_n |T|^{\frac12} P_n$ (as in Lemma~\ref{lemma:lph-uniform-completion-of-matc}) are elements in $C_0(X, \sch{p}{H})$ and, for $m \geq n$,
% \begin{align*}
%   \norm{T_m - T_n}_p^{2p} &\leq 2^p \sup_{x \in X} \tr |T(x)|^p \sum_{i=n+1}^m \ip{e_i, |T(x)|^p + |T^*(x)|^p e_i}
% \end{align*}
% \end{proposition}
% \begin{proof}
%   Note that $P_n S P_n |T|^{\frac12} P_n \in \cbstr(X, B(P_n H))$. Since the norm and strong topologies on $B(P_n H)$ coincide by finite-dimensionality, $T_n$ is norm continuous and as $\norm{T_n(x)}^p \leq \tr |T(x)|^p$ we have $T_n \in C_0(X, B(P_n H))$.
%   Thus, by Lemma~\ref{lemma:trivial-uniformly-finite-rank-so-schatten} we have $T_n \in C_0(X, \sch{p}{H})$.
% Then, the estimate on $\norm{T_m - T_n}_p^{2p}$ is a verbatim application of Lemma~\ref{lemma:lph-uniform-completion-of-matc} to their pointwise Schatten norms, together with the simple estimate $\norm{P_m |T(x)|^p P_m}_1 \leq \norm{|T(x)|^p}_1$.
% \end{proof}

\begin{theorem} \label{theorem:lp-iff-tr-tp-in-c0x}
  Let $T \in \cbstr{(X, B(H))}$.
  Then $T \in C_0(X, \sch{p}{H})$ if and only if $T \in \sch{p}{C_0(X,H)}$.
\end{theorem}
\begin{proof}
  The implication $\Rightarrow$ was already established in Remark ~\ref{remark:topological-schatten-class-in-operator-schatten-class}. % for the latter inclusion)follows from Proposition~\ref{proposition:c0-x-y-is-banach}, as discussed above.
  For the converse, assume that $T \in \cbstr{(X, B(H))}$ and that $x \mapsto \tr |T(x)|^p$ lies in $C_0(X)$.
  Then, pick an orthonormal basis $\{e_i\}_i$ of $H$ and write again $T_n = P_n T P_n$ as in Proposition~\ref{proposition:lph-uniform-completion-of-matc}. % by $T_n$.
  Since $x \mapsto \tr |T(x)|^p = \tr |T^*(x)|^p$ is in $C_0(X)$, the series $\sum_{i=1}^\infty \ip{e_i, |T(x)|^p e_i}$ and $\sum_{i=1}^\infty \ip{e_i, |T^*(x)|^p e_i}$ must converge uniformly on compact subsets of $X$ by Dini's theorem.
  That, in turn, implies that %for all $\epsilon$ there is a $n \in \N$ with
  $$
  \sup_{\mathclap{x \in X}} \norm{T_{n+k}(x) - T_n(x)}_p < \left[ 2^{2p-1} \sup_{\mathclap{x \in X}} \tr |T(x)|^p \sum_{\mathclap{i=n+1}}^{\mathclap{n+k}} \ip{e_i, \left(|T(x)|^p + |T^*(x)|^p\right) e_i} \right]^{\frac{1}{2p}}%< \epsilon,
  $$
goes to zero for large $n$. %for all $k \geq 0$.
  Since $C_0(X, \sch{p}{H})$ is a Banach space, the sequence $T_n$ thus converges (to $T$) in Schatten $p$-norm, so that $T \in C_0(X, \sch{p}{H})$.
\end{proof}

\begin{remark}
  In terms of tensor products of Banach spaces, Theorem~\ref{theorem:lp-iff-tr-tp-in-c0x} translates to the statement that $\sch{p}{H \otimes_\C A} \simeq \sch{p}{H} \otimes_{\epsilon} A$, the injective tensor product.
\end{remark}

\begin{corollary} \label{corollary:standard-continous-schatten-classes-form-ideal}
The continuous Schatten class $C_0(X, \sch{p}{H})$ forms a two-sided ideal in $\cbstr{(X, B(H))}$.
\end{corollary}
\begin{proof}
  Let $T \in C_0(X,\sch{p}{H})$ and let $T' \in \cbstr{(X, B(H))}$.
  Then, for any basis $\{e_i\}_i$ of $H$ the operators $T_n \eqdef P_n T P_n$ ($p \geq 2$) or $T_n \eqdef P_n S P_n |T|^{\frac12} P_n$ ($p < 2$), as in Proposition~\ref{proposition:lph-uniform-completion-of-matc}, converge to $T$ in the continuous Schatten $p$-norm.
  Now, note that $\norm{T' (T_m - T_n)}_p \leq \norm{T'} \norm{T_m - T_n}_p$ and $\norm{(T_m - T_n)T'}_p \leq \norm{T'} \norm{T_m - T_n}_p$; thus, $T' T_n$ and $T_n T'$ converge to $T' T$ and $T T'$, respectively, in the norm of $C_0(X, \sch{p}{H})$.
\end{proof}

\begin{remark}
  This result does not follow directly from the fact that $\sch{p}{H}$ is an ideal of $B(H)$ equipped with a Banach norm such that the inclusion into $B(H)$ is continuous: each such ideal induces a two-sided ideal $C_b(X, I) \subset C_b(X, B(H))$, but that is not necessarily an ideal of $\cbstr(X, B(H))$.
  An easy counterexample is given by $I = B(H)$ itself.
\end{remark}

\begin{corollary} \label{corollary:c0x-schp-compact}
  The continuous Schatten class $C_0(X, \sch{p}{H})$ is contained in the compact operators on the Hilbert $C^*$-module $C_0(X,H)$.
\end{corollary}
\begin{proof} The operators $T_n$ of Proposition~\ref{proposition:lph-uniform-completion-of-matc} are of finite rank in the Hilbert module sense, because they are contained in $C_0(X, B(V))$ for some finite-dimensional $V \subset H$ so that we can apply Lemma~\ref{lemma:trivial-uniformly-finite-rank-so-schatten}.
  As $T_n \to T$ for $T \in C_0(X, \sch{p}{H})$ in the Schatten $p$-norm, $T_n \to T$ in operator norm as well.
  We conclude that $T$ is compact in the Hilbert module sense.
\end{proof}

The following slight strengthening of Corollary~\ref{corolllary:s-in-Lp-if-sp-leq-tp-in-Lp} to $C_0(X, \sch{p}{H})$ now translates to $\sch{p}{C_0(X,H)}$.

\begin{corollary} \label{corollary:lp-order-ideal-in-compacts}
  If $0 \leq S \leq T \in \cbstr{(X,B(H))}$ and $T \in \sch{p}{C_0(X,H)}$ and additionally we have $S \in C_b^{\textup{norm}}(X, B(H))$, then $S \in \sch{p}{C_0(X,H)}$ and $\tr |S|^p \leq \tr |T|^p$.
\end{corollary}
\begin{proof}
  The operators $S, T$ are pointwise compact, norm continuous (this is where we use the additional assumption on $S$) and positive, so that their individual eigenvalues $\lambda_k$ (ordered decreasingly) are continuous (\cite[IV.3.5]{MR1335452}, see also Lemma~\ref{lemma:eigenvalues-individually-continuous} below).

  Then, by the min-max theorem, the $k$'th singular value of $\chi_* S$ is bounded by the $k$'th singular value of $\chi_* T$, so that the same holds for their $p$'th powers. As the Schatten norm of $\chi_* T$ is the sum of those $p$'th powers, which converges to a continuous function, the convergence must be uniform by Dini's theorem.
  Thus, the series $\sum_k \lambda_k(T)^p$ of elements of $A$ is Cauchy, so that the series $\sum_k \lambda_k(S)^p$ must be Cauchy as well.
  We conclude that $x \mapsto \norm{\chi_* S}_p$ lies in $A$.
\end{proof}

\begin{remark}
  The additional assumption in Corollary~\ref{corollary:lp-order-ideal-in-compacts} is necessary because the positive compact operators on a Hilbert $C^*$-module, in contrast to those on a Hilbert space, may not necessarily form an order ideal: there is an additional continuity requirement on the (pointwise compact) localizations.
  In contrast, as in Corollary~\ref{corolllary:s-in-Lp-if-sp-leq-tp-in-Lp}, the positive trace-class operators on a Hilbert $C^*$-module \emph{do} form an order ideal.
\end{remark}

\subsection{Properties of the Schatten classes on Hilbert modules}
We now return to the general setup of countably generated Hilbert $C^*$-modules over commutative $C^*$-algebras. For the case of the standard module $l^2(A)$ with $A= C_0(X)$ Theorem~\ref{theorem:lp-iff-tr-tp-in-c0x} shows that $\sch{p}{l^2(A)}$ is a Banach space and a two-sided ideal of $\lnmp{l^2(A)}$ that is moreover contained in $\cptmp{l^2(A)}$.
These are very desirable properties for general, countably generated Hilbert $A$-modules. 
Fortuitously, the existence of frames (Proposition~\ref{proposition:existence-of-frames}) and the pull-back criterion of Theorem~\ref{theorem:nonstandard-schatten-class-is-pullback-by-frame} allows us to easily establish equivalent properties of $\sch{p}{E_A}$ for all countably generated Hilbert $A$-modules $E_A$.

\begin{theorem} \label{theorem:schp-is-ideal-and-norm-has-properties}
  The space $\sch{p}{E_A}$ is a two-sided ideal of $\lnmp{E_A}$ that is contained in $\cptmp{E_A}$.
\end{theorem}
\begin{proof}
  Choose a frame $\fram$ of $E_A$ and let $\phi_\fram$ be the $\ast$-homomorphism induced by the frame transform: $\phi_\fram \colon \lnmp{E_A} \to \lnmp{l^2(A)}, T \mapsto \theta_\fram T \theta_\fram^*$.
  By Theorem~\ref{theorem:nonstandard-schatten-class-is-pullback-by-frame}, we have $T \in \sch{p}{E_A}$ if and only if $\phi_\fram (T) \in \sch{p}{l^2(A)}$.
  By Theorem~\ref{theorem:lp-iff-tr-tp-in-c0x}, then, $\sch{p}{E_A}$ is closed under finite linear combinations.
  Moreover, for $S \in \lnmp{E_A}, T \in \sch{p}{E_A}$, we have $\phi_\fram(ST) \in \sch{p}{l^2(A)}$ and $\phi_\fram(TS) \in \sch{p}{l^2(A)}$ by Corollary~\ref{corollary:standard-continous-schatten-classes-form-ideal}.
  We conclude that $\sch{p}{E_A}$ is a two-sided ideal.

  Moreover, as in the previous paragraph, $T \in \sch{p}{E_A}$ iff $\theta_\fram T \theta_\fram^* \in \sch{p}{l^2(A)} \subset \cptmp{l^2(A)}$ by Corollary~\ref{corollary:c0x-schp-compact}. But if $\theta_\fram T \theta_\fram^*$ is compact, then the operator $T = \theta_\fram^* \theta_\fram T \theta_\fram^* \theta_\fram$ is compact as well (essentially because $\theta_\fram^*( \ketbra{e_i}{e_j})\theta_\fram = \ketbra{\fram_i}{\fram_j}$ for any $i,j \in \N$). 
\end{proof}

\begin{remark}
\label{remark:comparison-to-dixmier-continuous-trace}
Dixmier's definition of \emph{continuous-trace} \cstar-algebras \cite[Chapter 4.5]{MR0458185} applies to $\cptmp{E_A}$.
By Theorem~\ref{theorem:schp-is-ideal-and-norm-has-properties}, the corresponding trace and Hilbert-Schmidt classes agree with our $\sch{1}{E_A}$ and $\sch{2}{E_A}$ respectively.
The projections satisfying Fell's criterion are then the \emph{compact} finite-rank ones, that is, those that correspond to finitely generated projective modules.
\end{remark}

\begin{remark}
  In the light of Theorem~\ref{theorem:schp-is-ideal-and-norm-has-properties}, we have obtained an \emph{a fortiori} method of determine whether a positive operator $T$ on a Hilbert \cstar-module over an abelian $C^*$-algebra is compact: it suffices that $\tr T^p$ lies in $A$ for some $p \geq 1$.
  Compare e.g. the proof of \cite[Proposition 7]{kaad2017factorization} and [ibid., Remark 6] to see that showing such compactness directly can be a nontrivial undertaking.
\end{remark}

Using Theorem~\ref{theorem:nonstandard-schatten-class-is-pullback-by-frame}, we can pull back the Banach norm on $C_0(X, \sch{p}{H}) \simeq \sch{p}{H \otimes_\C A}$ to $\sch{p}{E_A}$, and this turns out rather well:

\begin{theorem}
  \label{thm:norm-prop-Lp}
  The function $\norm{\cdot}_p \colon T \mapsto \norm{\tr |T|^p}_A^{1/p}$ is a norm that turns $\sch{p}{E_A}$ into a normed vector space.
  Moreover, for all $T \in \sch{p}{E_A}$,
  \begin{enumerate}
  \item $\norm{T}_p = \sup_{\chi \in \gfdual{A}} \norm{\chi_* T}_p$
  \item $\norm{T^*}_p = \norm{T}_p$
  \item $\norm{ST}_p \leq \norm{S} \norm{T}_p$ for all $S \in \lnmp{E_A}$ \label{inenum:ideal-norm}
  \item $\norm{T} \leq \norm{T}_p$
  \item For $p, q, r \geq 1$, if $S \in \sch{q}{E_A}$ and $T \in \sch{p}{E_A}$ with $\frac1{p} + \frac1{q} = 1/r$ then $ST \in \sch{r}{E_A}$ and $\norm{ST}_{r} \leq \norm{S}_q \norm{T}_p$. \label{inenum:hoelder}
  \end{enumerate}
Moreover, $\sch{p}{E_A}$ is a Banach space.
\end{theorem}
\begin{proof}
  Let $\fram$ be a frame of $E_A$ and consider the projection
  $P_\fram \eqdef \theta_\fram \theta_\fram^*$ in $\lnmp{l^2(A)}$.
  Recall that the invertible $\ast$-homomorphism
  $\phi_\fram \colon T \mapsto \theta_\fram T \theta_\fram^*$ maps
  $\sch{p}{E_A}$ into $P_\fram \sch{p}{l^2(A)} P_\fram$ by
  Theorem~\ref{theorem:nonstandard-schatten-class-is-pullback-by-frame}.
  Moreover, elementary calculation shows that $\phi_\fram \mid_{\sch{p}{E_A}}$ is an isomorphism of normed spaces.

  For the properties of the norm, recall that $\chi(\tr |T|^p) = \tr |\chi_* T|^p = \norm{\chi_* T}_p^p$.
  Thus, $\norm{\tr |T|^p}_A^{1/p} = \sup_{\chi \in \gfdual{A}} \norm{\chi_* T}_p^{1/p}$.
  Therefore, by the analogous properties of $\sch{p}{H}$, we see that $\norm{T^*}_p = \norm{T}_p$, $\norm{T}_p \geq \norm{T} = \sup_{\chi \in \gfdual{A}} \norm{\chi_* T}$ and $\norm{ST}_p \leq \sup_{\chi \in \gfdual{A}} \norm{\chi_*(S)} \norm{\chi_* T}_p \leq \norm{S} \norm{T}_p$.

  For property \ref{inenum:hoelder}, recall that $\theta_\fram ST \theta_\fram^* = \theta_\fram S \theta_\fram^* \theta_\fram T \theta_\fram^* \in \sch{r}{l^2(A)}$ by Proposition~\ref{proposition:continuous-hoelder-von-neumann-inequality}, and so $ST \in \sch{r}{E_A}$ by Theorem~\ref{theorem:nonstandard-schatten-class-is-pullback-by-frame}.
  The norm inequality then follows from the H\"older--von Neumann inequality for operators on Hilbert spaces.

  Finally, to establish completeness of $\sch{p}{E_A}$ it is enough to prove that its pullback $P_\fram \sch{p}{l^2(A)} P_\fram$ is a closed subspace of $\sch{p}{l^2(A)}$. But if $P_\fram T_n P_\fram \to T$ in $\sch{p}{l^2(A)}$ then
  $$
  \| P_\fram T_n P_\fram  - P_\fram T P_\fram  \|_p =   \| P_\fram (P_\fram T_n P_\fram  - T)P_\fram \|_p \leq \| P_\fram\|^2  \| (P_\fram T_n P_\fram  - T)\|_p \to 0
  $$
  as $n \to \infty$, in virtue of the just-proved inequality \ref{inenum:ideal-norm}. Hence, $P_\fram T_n P_\fram$ converges to $P_\fram T P_\fram \in P_\fram \sch{p}{l^2(A)} P_\fram$ as desired.
\end{proof}

\subsection{The Hilbert module of Hilbert--Schmidt operators}
\label{section:hilbert-schmidt}
The Hilbert--\linebreak[4]Schmidt class $\sch{2}{E_A}$ is a somewhat special case among the Schatten classes, because the map $T \mapsto \tr T^*T$ is a positive definite quadratic form.
That is, it induces an inner product as we will now explore.

\begin{definition}\label{definition:hilbert-schmidt-inner-product}
  The pairing $\ip{\cdot, \cdot}_2 \colon \sch{2}{E_A} \times \sch{2}{E_A} \to A$ is given by
  \[
    \ip{S, T}_2 \eqdef \frac 14 \sum_{k \in \Z / 4 \Z} i^k \tr |T + i^k S|^2
  \]
\end{definition}

When viewed fiberwise, this is just the ordinary Hilbert--Schmidt inner product:

\begin{proposition} \label{proposition:hilbert-schmidt-inner-product-is-fiberwise}
  For $S, T \in \sch{2}{E_A}$ and a character $\chi$ of $A$ the pairing $\ip{S, T}_2$ satisfies $\chi \left(\ip{S, T}_2 \right) = \tr( %\ip{
  (\chi_* S)^* \chi_* T)$. %}_{\sch{2}{\chi _* E_A}}$.
  Moreover, if $\fram$ is a frame of $E_A$, then the series $\sum_{i=1}^\infty \ip{S \fram_i, T \fram_i}$ converges in norm to $\ip{S, T}_2$.
\end{proposition}
\begin{proof}
  Since $\chi_*$ is a homomorphism, the first part follows from the fact that $\tr \chi_*( |T + i^k S|^2 ) = \chi(\tr |T + i^k S|^2)$ by the polarization identity for the fiberwise Hilbert--Schmidt inner product.
  Since
  $$\ip{S \fram_i, T \fram_i} = \ip{\fram_i, S^* T \fram_i} = \frac 14\sum_{k \in \Z /4 \Z} i^k \ip{\fram_i, |T + i^k S|^2 \fram_i},
  $$
  the second part follows from Theorem~\ref{theorem:nonstandard-schatten-class-is-pullback-by-frame}.
\end{proof}

\begin{corollary}
  The pairing $\langle \cdot,\cdot \rangle_2$ on $\sch{2}{E_A}$ is non-degenerate and sesquilinear.
\end{corollary}
%% \begin{proof}
%%   This follows directly from Proposition~\ref{proposition:hilbert-schmidt-inner-product-is-fiberwise}.
%% \end{proof}

Next, because $A$ is commutative $E_A$ is automatically an $A$-bimodule. In fact, there is a $*$-homomorphism $\rho \colon A \to \lnmp{E_A}$ given by $\rho(a)(v) \eqdef v \cdot a$. This makes $\lnmp{E_A}$ an $A$-bimodule with $a \cdot T \cdot b= \rho(a) \circ T \circ \rho(b)$ for all $a,b \in A$ and $T \in \lnmp{E_A}$. 
\begin{proposition}
  The $A$-bimodule structure of $\lnmp{E_A}$ restricts to $\sch{p}{E_A}$ for all $1 \leq p < \infty$ and satisfies $\norm{T \cdot a}_p \leq \norm{a} \norm{T}_p$ as well as $\norm{a \cdot T}_p \leq \norm{a} \norm{T}_p$.
\end{proposition}
\begin{proof}
Since any $*$-homomorphism between $C^*$-algebras is norm decreasing, this follows from Theorem~\ref{theorem:schp-is-ideal-and-norm-has-properties} since
% We have $\norm{\ip{v \cdot a, v \cdot a}} = \norm{a^* \ip{v, v} a} \leq \norm{a}^2 \norm{\ip{v, v}}$, so that $\norm{\rho(a)} \leq \norm{a}$.
%Thus,
$\norm{T \rho(a)}_p \leq \norm{\rho(a)} \norm{T}_p \leq \norm{a} \norm{T}_p$.% by Theorem~\ref{theorem:schp-is-ideal-and-norm-has-properties}.
\end{proof}

All this leads to the following result for the case that $p=2$:

\begin{proposition}
  With the above right $A$-action %of Definition~\ref{definition:right-a-action-on-lnmp}
  and the inner product of Definition~\ref{definition:hilbert-schmidt-inner-product}, $\sch{2}{E_A}$ becomes a Hilbert $A$-module.
\end{proposition}
\begin{proof}
  Note that $\chi_* (T \circ \rho(a)) = \chi(a) \chi_* T$ for all $a \in A$ and characters $\chi$ of $A$.
  Thus, with Proposition~\ref{proposition:hilbert-schmidt-inner-product-is-fiberwise}, the inner product is $A$-sesquilinear.
  All that is left to show, therefore, is that $\sch{2}{E_A}$ is complete.
  That, however, was proven already in Theorem~\ref{theorem:schp-is-ideal-and-norm-has-properties}.
\end{proof}

\begin{proposition}
  Let $H$ be a separable Hilbert space.
  Then, $\sch{2}{H \otimes_\C A}$ is isomorphic, as a Hilbert $A$-module, to $\sch2{H} \otimes_\C A$.
\end{proposition}
\begin{proof}
  Under the isomorphism $H \otimes_\C A \simeq C_0(X,H)$, $\sch{2}{H \otimes_\C A}$ is mapped isometrically onto $C_0(X, \sch{2}{H})$ by Theorem~\ref{theorem:lp-iff-tr-tp-in-c0x}.
  Under this identification, the Hilbert $C^*$-module structure of $\sch{2}{H \otimes_\C A}$ coincides with the canonical Hilbert $C^*$-module structure on $C_0(X, \sch{2}{H})$ induced by the inner product on the Hilbert space $\sch{2}{H}$.
  Thus, we have $\sch{2}{H \otimes_\C A} \simeq C_0(X, \sch{2}{H})$ as Hilbert $C^*$-modules. 
  Now, invoke once again the isomorphism $\sch{2}{H} \otimes_\C A \simeq C_0(X, \sch{2}{H})$, this time for the Hilbert space $\sch{2}{H}$, to complete the proof.
  %see that the map $m \colon \sch2{H} \otimes_\C C_0(X) \to C_0(X, \sch2{H})$, $T \otimes a \mapsto Ta$, is an isomorphism.
\end{proof}

\subsection{The trace class and the trace}

We will refer to the ideal $\sch1{E_A}\subset  \lnmp{E_A}$ consisting of those operators $T $ for which $\tr |T|$ is given by an element of $A$, as the \emph{trace class}.
In the case of Schatten classes of Hilbert spaces, {\em i.e.} $A = \C$ and $E_A = H$, it is customary to identify the trace class as the ideal generated by squares of elements of the Hilbert--Schmidt class, in order to relate the Hilbert--Schmidt inner product to a linear function, the \emph{trace}, on $\sch{1}{H}$.
The situation here is completely analogous:

\begin{proposition}
  Let $\sch{2}{E_A} \sch{2}{E_A}$ denote the set $\{ RS \mid R \in \sch{2}{E_A}, S \in \sch{2}{E_A} \} \subset \lnmp{E_A}$.
  Then, $\sch{2}{E_A} \sch{2}{E_A} = \sch{1}{E_A}$.
\end{proposition}
\begin{proof}
  The inclusion $\sch{2}{E_A} \sch{2}{E_A} \subset \sch{1}{E_A}$ is a direct consequence of the H\"older--von Neumann inequality in Theorem~\ref{thm:norm-prop-Lp}(\ref{inenum:hoelder})%~\ref{theorem:schp-is-ideal-and-norm-has-properties}.
  Conversely, let $T \in \sch{1}{E_A}$.
  Then, $T = S |T|^{\frac12}$ in the usual weak polar decomposition, with $|S| = |T|^{\frac12}$.
  By Proposition~\ref{proposition:t-in-lp-iff-tabs-in-lp-iff-tp-in-l1}, $S$ and $|T|^{\frac12}$ lie in $\sch{2}{E_A}$.
\end{proof}

This furnishes us with a way to turn the bilinear map $\ip{\cdot, \cdot}_2$  on $\sch{2}{E_A}$ into a linear map $\tr$ on $\sch1{E_A}$ called the \emph{trace}:

\begin{definition}
  The \emph{trace} on $\sch{1}{E_A}$ is the map $\tr \colon T \mapsto \langle {S^*, |T|^{\frac12}}\rangle _2$, where $T = S |T|^{\frac12}$ is the weak polar decomposition.
\end{definition}

\begin{proposition}
  The trace is well-defined.
  Moreover, let $\fram$ be a frame.
  Then, the series $\sum_i \ip{\fram_i, T \fram_i}$ converges in norm to $\tr T$.
\end{proposition}
\begin{proof}
  Assume $T = RS$ with $R, S \in \sch2{E_A}$ and let $\fram$ be a frame.
  Then $\ip{R^*, S}_2 = \sum_{i} \ip{R^* \fram_i, S \fram_i}$, which  converges in norm by Proposition~\ref{proposition:hilbert-schmidt-inner-product-is-fiberwise}.
  As $\ip{R^* \fram_i, S \fram_i} = \ip{\fram_i, T \fram_i}$, we have two expressions for $\tr T$: one independent of the decomposition $T = RS$ and one independent of the choice $\fram$ of frame.
  The proposition follows.
\end{proof}

\begin{corollary} \label{corollary:trace-is-pointwise}
  Let $\chi$ be a character of $A$ and let $T \in \sch{1}{E_A}$.
  Then, $\tr \chi_* T = \chi(\tr T)$.
\end{corollary}
\begin{proof}
  Note that $\chi(\ip{\fram_i, T \fram_i}) = \ip{\chi_* \fram_i, \chi_* T \chi_* \fram_i}$. Since $\chi_* \fram_i$ is a frame of $\chi_* E_A$ the result then follows from Corollary~\ref{corollary:frame-trace-is-trace}.
\end{proof}

\begin{corollary}
  For $T \in \sch{1}{E_A}$, $|\tr T| \leq \tr |T|$, as elements of $A$.
  In particular, if $A$ is unital $| \tr T | \leq \tr |T| \leq \norm{T}_1 1_A$.
\end{corollary}
\begin{proof}
  For all $\chi \in \gfdual{A}$ we have $\chi ( | \tr T|) = | \tr \chi_* T | \leq %\norm{\chi_* T}_1 =
  \tr |\chi_* T| = \chi( \tr |T|)$ by the inequality $|\tr S| \leq \tr |S|$ on $\sch{1}{H}$ for Hilbert spaces $H$.
  As the characters separate $A$, the first statement follows immediately.
  The last statement follows from the inequality $ a \leq \norm{a} 1_A$ for positive elements of any unital \cstar-algebra.
\end{proof}

Now, we can finally show that the trace is cyclic.
The standard approach is as follows:

\begin{proposition} \label{proposition:trace-is-cyclic}
  If $S, T \in \lnmp{E_A}$ are such that $ST \in \sch{1}{E_A}$ and $TS \in \sch{1}{E_A}$, then $\tr ST = \tr TS$.
  % Let $S, T \in \sch2{E_A}$ or let $S \in \lnmp{E_A}$ and $T \in \sch1{E_A}$.
%  Then, $\tr S T = \tr T S$.
\end{proposition}
\begin{proof}
  Consider the value of a character $\chi \in \widehat{A}$ on the difference $\tr S T - \tr T S \in A$ and use Corollary \ref{corollary:trace-is-pointwise} above:
  \begin{align*}
    \chi(\tr S T - \tr T S) &=\chi(\tr (ST)) - \chi(\tr (TS))\\
    &=  \tr (\chi_* (ST)) - \tr (\chi_*(TS)) \\
    &=  \tr (\chi_* (S)\chi_*(T)) - \tr (\chi_*(T)\chi_*(S))
  \end{align*}
  We may now use the tracial property of the trace on $\sch1{\chi_*E_A}$ ({\em cf.} \cite[Corollary 3.8]{MR2154153}) and the fact that $\chi$ separates $A$ to conclude the proof. 
 %%  For the first part, note that $\tr T^* T = \tr T T^*$ for $T \in \sch{2}{E_A}$ because $\norm{\chi_* T}_p = \norm{\chi_* T^*}_p$ for all characters $x$ of $A$ -- in the case $p = 2$ this can be seen even without recourse to the singular value decomposition by recalling Lemma~\ref{lemma:l2-fram-framt-bound}.
%% Thus, 
%%   \begin{align*}
%%     4 \tr ST &= \sum_{k \in \Z / 4 \Z} i^k \ip{S + i^k T, S + i^k T} \\
%%              &= \sum_{k \in \Z / 4 \Z} i^k \ip{S^* + i^{-k} T^*, S^* + i^k T^*} \\
%%              &= \sum_{k \in \Z / 4 \Z} i^k \ip{T^* + i^k S^*, T^* + i^k S^*} = 4 \tr TS.
%%   \end{align*}
%%   For the second part, assume $T = R^* R$ with $R \in \sch2{E_A}$.
%%   Then, $\tr ST = \tr (SR^*) R = \tr R (S R^*) = \tr (RS) R^* = \tr R^* R S$.
\end{proof}

%% Proposition~\ref{proposition:trace-is-cyclic} can be improved upon slightly by using more advanced knowledge of the pointwise trace classes.

%% \begin{proposition}
%%   If $S, T \in \lnmp{E_A}$ are such that $ST \in \sch{1}{E_A}$ and $TS \in \sch{1}{E_A}$, then $\tr ST = \tr TS$.
%% \end{proposition}
%% \begin{proof}
%%   The result is true whenever $E$ is a Hilbert space, see e.g. \cite[Corollary 3.8]{MR2154153}.
%%   Corollary~\ref{corollary:trace-is-pointwise} completes the proof.
%% \end{proof}

\section{Applications of Schatten classes}
\subsection{The Fredholm determinant}
\label{section:fredholm-determinant}
As a first application of the above theory of Schatten classes for Hilbert modules over unital abelian $C^*$-algebras, we consider the Fredholm determinant. Let $A$ be a %\emph{unital}
commutative \cstar-algebra and let $E_A$ be a countably generated $A$-module.

\begin{definition}
  Let $G(E_A) \subset \lnmp{E_A}$ be the set of bounded, invertible endomorphisms of $E_A$ of the form $\id_{E_A} + T$, where $T \in \sch{1}{E_A}$.
\end{definition}

Note that $G(E_A)$ is a group under the multiplication of $\lnmp{E_A}$ because we have $(\id + T)^{-1} = \id  - T(\id + T)^{-1}$ and $\sch{1}{E_A}$ is an ideal.
We will define the Fredholm determinant, first on the standard module, and then by pullback by a frame transform on general countably generated Hilbert $A$-modules.

The following definition of the Fredholm determinant on a Hilbert space $H$ is well-known.
See e.g. \cite[Chapter 3]{MR2154153} for a brief discussion in the context of Lidskii's theorem.

\begin{definition} \label{definition:fredholm-determinant-on-hilbert-space}
  Let $T \in \sch{1}{H_\C}$.
Then, the \emph{Fredholm determinant} of $\id + T$ is
  \[
    \det(\id + T) \eqdef \sum_{k=0}^\infty \tr \exterior^k T.
  \]
\end{definition}

Recall that the series converges by the estimate $\| \exterior^k T \|_1 \leq \|T\|_1^k / k!$.

\begin{remark} \label{remark:fredholm-determinant-is-continuous}
  We have $| \det(\id +T_1) - \det(\id + T_2) | \leq \norm{T_1 - T_2}_1 \exp(\norm{T_1}_1 + \norm{T_2}_1 + 1)$, cf. \cite[Theorem 3.4]{MR2154153}. Thus, $T \mapsto \det(\id + T)$ is a continuous function on $\sch{1}{H}$.
\end{remark}

The Fredholm determinant of $\id + T$ is invariant under conjugation of $T$ by partial isometries that commute with $T$:

\begin{lemma}  \label{lemma:fredholm-determinant-invariant-under-commuting-partial-isometries}
  If $u \colon H \to K$ is a partial isometry of Hilbert spaces, and $T \in \sch{1}{H}$ is such that $u^* u T = T u^* u = T$, then $\id + u T u^* \in G(K_\C)$ and in fact $ \det(\id_K + u T u^*) = \det(\id_H + T)$.
%  Moreover, if $\id + T \in G(K_\C)$, then $\id + \theta_\fram T \theta_\fram^* \in G(H_\C)$.
\end{lemma}
\begin{proof}

Note that we have
  \[
    (\id + u T u^*)(\id - u T(\id + T)^{-1} u^*) = (\id + uT u^*) - u(\id + T)T(\id + T)^{-1} u^* = \id,
  \]
  so that $\id + uT u^* \in G(K_\C)$.

Next, note that $\tr \exterior^k uTu^* = \tr \exterior^k u^* u T = \tr \exterior^k T$ so that indeed $\det(\id_H + T) = \sum_{k=0}^\infty \tr \exterior^k T = \sum_{k=0}^\infty \tr \exterior^k u T u^* = \det(\id_K + uTu^*)$.
 % Next, if $v$ is an eigenvector of $T$ of nonzero eigenvalue $\lambda$, then $u (\lambda v) = u (Tv) = u T u^* (uv)$ so that $uv$ is an eigenvalue of $u T u^*$ of eigenvalue $\lambda$. Moreover, $u^* u (\lambda v ) = u^* u T v = T v = \lambda v$ showing that $u^* u v = v$. Conversely, if $w$ is a eigenvector of $u T u^*$ of nonzero eigenvalue $\lambda$ then $T(u^* w) = u^* u T u^* w = \lambda u^* w$ so that $u^* w$ is an eigenvalue of $T$ of eigenvalue $\lambda$. Moreover, $u u^* (\lambda w) = u u ^* u T u^* w = u T u ^* w = \lambda w$ so that $uu^* w = w$. 
  %% , then $u T u^* (w) = \lambda w$ so that $u u^* (\lambda w) = u u^* u T u^* w = u T u^* (w) = \lambda w$ and thus, $u u^* w = w$. Moreover, $u^* (\lambda w) = T u^* w$ so that $T(u^* w) = \lambda (u^* w)$.
  %% Conversely, if $v$ is an eigenvector of $T$ of eigenvalue $\lambda$, then $u (\lambda v) = u (Tv) = u T u^* (uv)$.
  % We conclude that the \emph{nonzero} spectrum of $uTu^*$ with multiplicities, equals that of $T$.
  % By the expression for the Fredholm determinant, the second part of the proposition follows.
\end{proof}

\begin{proposition} \label{proposition:fredholm-determinant-is-pullback-by-frame}
  If $K$ is a separable Hilbert space equipped with a frame $\fram$ and $T \in \sch{1}{K}$, then $\det(1 + \theta_\fram T \theta_\fram^*) = \det(1 + T)$ in terms of the corresponding frame transform $\theta_\fram \colon K \to l^2$.
\end{proposition}
\begin{proof}
Since $\theta_\fram^* \theta_\fram = \id_K$ commutes with $T$ we can apply Lemma~\ref{lemma:fredholm-determinant-invariant-under-commuting-partial-isometries}.
\end{proof}

\begin{definition} \label{definition:fredholm-determinant}
  Let $E_A$ be a countably generated Hilbert $C^*$-module over a unital and abelian \cstar-algebra $A$.
  For $T \in \sch{1}{E_A}$, the Fredholm determinant $\det(\id + T)$ of $\id + T$ is the function on $\widehat{A}$ given by $\chi \mapsto \det(\chi_*(\id + T))$.
\end{definition}

\begin{proposition}
 Let $A$ be unital and abelian as above. For $T \in \sch{1}{E_A}$, the Fredholm determinant lies in $A \equiv C(\widehat A)$ and as such we have $\chi(\det (\id +T)) = \det ( \chi_* (\id +  T))$.
\end{proposition}
\begin{proof}
  Let $\fram$ be a frame of $E_A$.
  Note that, for all $\chi \in \gfdual{A}$, $\det(\chi_*(\id + T)) = \det(\chi_*(\id + \theta_\fram T \theta_\fram^*))$ by Proposition~\ref{proposition:fredholm-determinant-is-pullback-by-frame}.
  That is, $\det(\id + T) = \det(\id + \theta_\fram T \theta_\fram^*)$.
  Now, $\theta_\fram T \theta_\fram^* \in \sch{1}{l^2(A)}$ by Theorem~\ref{theorem:nonstandard-schatten-class-is-pullback-by-frame}.
  With Remark~\ref{remark:fredholm-determinant-is-continuous} we see that $\chi \mapsto \det  (\id + \chi_*S)$ is continuous whenever $S \in \sch{1}{l^2(A)}$.
  Thus, since $A$ is unital (and thus $\widehat{A}$ compact) we find that $\det(\id + T) \in C(\widehat A) = A$.
\end{proof}

\begin{proposition}
  Let $T \in C(X, \sch{1}{H})$ with $X=\widehat{A}$.
  Then we have $\exterior^k(T) \in C\left(X, \sch{1}{\exterior^k H}\right)$. 
  In particular, one has $\det(\id + z T) = \sum_{k \geq 0} z^k \tr \exterior^k(T) $, and $z \mapsto \det(\id + z T)$ is entire (as an $A$-valued function on $\C$).
\end{proposition}
\begin{proof}
  Let $A, B \in \sch{1}{H}$ and note that $\exterior^{k+1}(A) - \exterior^{k+1}(B) = (\exterior^k(A) - \exterior^k(B)) \wedge A + \exterior^k(B) \wedge (A - B)$, which can be iterated to yield
  \[
\norm{\exterior^k(A) - \exterior^k(B)}_1 \leq \norm{A-B}_1 \sum_{m=0}^{k-1} \norm{A}_1^m \norm{B}_1^{k-1-m} .
\]
As a consequence, we see that $\exterior^k(T) \in C\left(X, \sch{1}{\exterior^k H} \right)$ whenever $T$ lies in $C(X, \sch{1}{H})$.

Moreover, we have the pointwise series expression
$$
\det (\chi_*(\id + z T))= \det (\id + z ~ \chi_* T)  = \sum_{k \geq 0} z^k \tr \exterior^k(\chi_*T)
$$
as in \cite[Lemma 3.3]{MR2154153}. Since $\tr \exterior^k(\chi_* T) \leq \norm{\chi_* T}_1^k / k!$ the series $\sum_{k \geq 0} z^k \tr \exterior^k(T)$ in $A$ converges absolutely for all $z \in \C$.
This implies in particular that $z \mapsto \det(\id + z T)$ is entire.
\end{proof}

\begin{remark}
  If $f \in A$ is invertible, then $\det(\id - f^{-1} T) = 0$ in $A$ if and only if $\chi(f)=f(x)$ is a (nonzero) eigenvalue of $\chi_* T$ for all $\chi \in \widehat{A}$ (corresponding to the point $x \in X$).
\end{remark}

\begin{proposition} \label{proposition:fredholm-determinant-is-homomorphism}
  The Fredholm determinant is multiplicative in the sense that $\det(\id + T)(\id + S) = \det(\id + T) \det(\id + S)$.
\end{proposition}
\begin{proof}
  This follows simply from the analogous property of the Fredholm determinant on Hilbert spaces, since
  \begin{align*}
    \chi(\det((\id + T)(\id + S)) )&= \det (\chi_*((\id + T)(\id + S) )) \\
    &= \det (\chi_*(\id + T))   \det (\chi_* (\id + S) ))  \\
    &= \chi(\det(\id + T) \chi(\det(\id  + S)),
  \end{align*}
  for all $\id + T$, $\id + S$ in $G(E_A)$ and all $\chi \in \gfdual{A}$. 
\end{proof}

\begin{proposition}%[Exponential trace formula]
  For $0 \leq |z| < \norm{T}_1^{-1}$, the Fredholm determinant satisfies
  \[
    \det(\id + z T) = \exp\left(
      \sum_{n=1}^\infty \frac{(-1)^{n+1}}{n} z^n \tr T^n
      \right),
  \]
  and the series converges absolutely.
\end{proposition}
\begin{proof}
  Absolute convergence follows from the fact that $\norm{\tr T^n} \leq \norm{T^n}_1 \leq \norm{T}_1^n$, so that the sum of the norms of the summands is bounded by $\log (1 + z \norm{T}_1)$.
  
  As the characters separate $A$, it will suffice to prove the formula for the case $E_A = H_\C$, where it is well-known.
  For $|z| \norm{T}_1 < 1$, by absolute convergence of the trace and Lidskii's theorem,
  \begin{align*}
    \sum_{n=1}^\infty \frac{(-1)^{n+1}}{n} z^n \tr T^n &= \sum_{n=1}^\infty (-1)^{n+1} z^n \sum_{k} \frac{\lambda_k(T)^n}{n} \\
    &= \sum_{k} \log(1 + z \lambda_k(T)),
  \end{align*}
  so that the exponential of the right-hand side equals $\det(\id + z T)$ by Definition~\ref{definition:fredholm-determinant-on-hilbert-space}.
\end{proof}

\begin{proposition}
  Let $T \in \sch{1}{E_A}$.
  Then $\id + T$ is invertible (that is, $\id +T \in G(E_A)$) if and only if $\det (\id + T) \in A$ is invertible.
\end{proposition}
\begin{proof}
  By Proposition~\ref{proposition:fredholm-determinant-is-homomorphism}, $\det(\id + T)$ is invertible whenever $\id + T$ is.

For the converse statement, in view of Proposition~\ref{proposition:fredholm-determinant-is-pullback-by-frame} we may assume without loss of generality  that $E_A = l^2(A) \simeq C(X,H)$. So, if $\det(\id + T)$ is invertible then it follows that $\id + \chi_*T$ is invertible for all $\chi$ by \cite[Theorem 3.5b)]{MR2154153}.
  Moreover, as $\norm{S_1^{-1} - S_2^{-1}}_1 \leq \norm{S_1^{-1}} \norm{S_2^{-1}} \norm{S_1-S_2}_1$, the $B(H)$-valued map $\chi \mapsto (\id + \chi_*T)^{-1}$ lies in $C_b(X, B(H)))$ whenever it is bounded.

  Now, for all eigenvalues $\lambda \in \sigma(\chi_*T)$ we have $\det(\id - \lambda^{-1} \chi_*T) = 0$ so that $\det(\id - \lambda^{-1} T)$ is not invertible in $A$.
  Thus, if there exist a sequence $\lambda((\chi_i)_* T) \in \sigma((\chi_i)_* T)$ converging (in $\C$) to $-1$, the element $\det(\id + T) = \lim_{i} \det(\id - \lambda(x_i)^{-1} T)$ is contained in the closed set consisting of the non-invertible elements of $A$, which contradicts the assumpion.

We conclude that there exists $\mu > 0$ with  $\inf_n |1 + \lambda_n(\chi_* T)| > \mu$ uniformly for $\chi \in \widehat{A}$.
In particular, $\norm{(\id + \chi_* T)^{-1}} = \sup_n |1 + \lambda_n( \chi_* T)|^{-1} < 1 / \mu$ for all $x$.
We conclude that $\chi \mapsto (\id + \chi_*T)^{-1}$ is bounded and therefore continuous.
\end{proof}

\begin{corollary}
  The Fredholm determinant is a homomorphism from the group $G(E_A)$ to the group of invertible elements of $A$.
\end{corollary}

In particular, the map $\det$ extends the (matrix) determinant homomorphism $K_1^{\text{alg}}(A) = GL_\infty(A) / [GL_\infty(A), GL_\infty(A)] \to A$ to all of $G(l^2(A))$.

\begin{remark}
  In \cite{MR424786,MR432738} it was shown that the Fredholm determinant $\det (\id + z T)$ is the unique additive invariant of an endomorphism $T: E \to E$ on finitely generated projective $A$-modules. It is an interesting open question to see under which additional conditions this result extends to the countably generated Hilbert module context. Clearly, the Fredholm determinant discussed above gives rise to a additive map from a countably generated Hilbert $A$-module $E$ equipped with a trace-class operator $T$ to analytic $A$-valued functions. 
\end{remark}

\subsection{The zeta function}
\label{section:zeta-function}

As a second application we consider zeta functions associated to positive Schatten class operators on a Hilbert module. Again, $A $ is a commutative $C^*$-algebra and we identify $\widehat A =X$ so that $A \cong  C_0(X)$.

%% Let $(A, H, D)$ be a spectral triple (in the sense of \cite[Definition 9.16]{MR1789831}).
%% Such a triple is said to be $p$-summable, for $p \in [1, \infty)$, if $(1 + D^2)^{-1/2} \in \sch{p}{H}$, and the \emph{zeta function} of $D$ is given by the map $z \to \tr (1+D^2)^{-z/2}$. See Section~\ref{sec:summ-some-unbo} for further context and generalization to abelian Hilbert \cstar-modules.

%% Note that, in this example, the resolvent $T = (1+D^2)^{-1/2}$ satisfies $0 \leq T \leq 1$, so that the series defining the zeta function is absolutely convergent (at least pointwise in $z$, but cf. Lemma~\ref{lemma:continuous-jensen-cahen} below) whenever $\Re z > p$.
%% We will mimic this definition in the setting of Hilbert \cstar-modules.

\begin{definition}
  \label{definition:zeta-function}
  Let $0 \leq T \leq 1 \in \sch{p}{E_A}$, for $p \geq 1$.
  For $z \in \C$ with $\Re z > p$, define $T^z$ using the continuous functional calculus in $\lnmp{E_A}$. 
Then, the associated \emph{zeta function} is the function on the complex half-plane $\Re z > p$ given by
  \[
    \zeta(z, T) \eqdef \tr T^z.
  \]
\end{definition}

We will show that the zeta function is in fact \emph{holomorphic} (in the sense of Banach space-valued holomorphic functions, see e.g. \cite[Definition 3.30]{MR1157815}) on the defining half-plane.

First, let us recall %show (cf. \cite[IV.3.5]{MR1335452})
that the individual (fiberwise) eigenvalues of positive compact operators are themselves continuous. 

\begin{lemma}%[Continuity of finite parts of the spectrum] 
\label{lemma:eigenvalues-individually-continuous}
  Let $0 \leq T \in \cptmp{E_A}$.
  For $k \geq 0$ and $\chi \in \gfdual{A}$ be the character corresponding to the point $x \in X$, let $\{ \lambda_k(x)\}_k$ be the eigenvalues of $\chi_* T$, in decreasing order with multiplicity. %, where we set $\lambda_k(\chi) \eqdef 0$ if $k > \dim \chi_* E_A$.
  Then, the map $x \mapsto \lambda_k(x)$ is an element of $A$.
\end{lemma}
\begin{proof}
A proof can be found in \cite[Theorem 6.4.2]{MR2125398} (see also {\em op.cit.} Section 6.6). 
  \end{proof}

\begin{remark}
Note that it is important to know that both $T$ is pointwise compact and that $T$ is norm continuous. Dropping the latter assumption is fatal (see e.g. Example~\ref{example:locally-nontrivial-bundle-projection}).%, whereas dropping the former would restrict us to showing local continuity of compact summands of $T$ instead.
%The extension of the Lemma to nonpositive operators is somewhat more cumbersome to state, because the eigenvalues may cross $0$.

By self-adjointness and norm continuity, the full spectrum of $T$ is continuous in norm topology, cf. \cite[Remark IV.3.3]{MR1335452}.
However, the spectral projections (or even the eigenvectors) can in general \emph{not} be continuously extended over any open neighbourhood, even in the case where the module is finitely generated: see e.g. \cite{MR765586}. \emph{A fortiori}, continuous diagonalizability (`diagonability') is entirely out of the question in general.
\end{remark}

The classical treatment of zeta functions of operators on Hilbert spaces is in terms of the Dirichlet series $\tr T^z = \sum_k \lambda_k^z$.
The Jensen--Cahen theorem (see e.g. \cite{MR0185094}) shows that these series converge uniformly on angular regions contained in the defining half-plane, so that the limit is in fact holomorphic.
The theorem translates very well to the Hilbert module setting.

\begin{lemma}%[Continuous version of the Jensen--Cahen theorem]
  \label{lemma:continuous-jensen-cahen}
  Suppose that $0 \leq T \leq 1 \in \sch{p}{E_A}$. 
  For $0 < \alpha  < \pi/2$, denote the angular region $\{z \in \C \mid |\operatorname{Arg}(z - p)| \leq \alpha\}$ by $C_\alpha$.
  Then, for all $\epsilon > 0$ there exists $m_0 \geq 0$ such that,
  for all $n \geq m \geq n_0$ and all $z \in C_\alpha$,
  \[
    \norm{
      \sum_{k=m}^n \lambda_k^{z}
    }_A
    < \epsilon.
  \]
\end{lemma}
\begin{proof}
  The proof is based on \cite[Theorem 2]{MR0185094}. 
  Consider the series $\sum_{k=m}^n \lambda_k^p \lambda_k^{z-p}$.
  Write $A(p, q) \eqdef \sum_{k=p}^q \lambda_k^p$ and $\Del_{z,k} \eqdef \lambda_{k+1}^{z} - \lambda_k^{z}$.
  Then, by Abel's lemma on partial summation \cite{MR1577619}, we have
  \[
    \sum_{k=m}^n \lambda_k^z = \sum_{k=m}^{n-1} A(m, k) \Del_{z,k} + A(m, n) \lambda_n^{z}
  \]
  Now, because the series $\sum_{k=0}^\infty \lambda_k(x)^p$ converges pointwise to $\tr T^p$, Dini's theorem shows that $\sum_{k=0}^\infty \lambda_k^p$ converges in norm.
  In particular, for all $\epsilon > 0$ there exists $m_0$ such that $\norm{A(m, q)} < \epsilon \cos \alpha$ for all $q \geq m \geq m_0$.

  By \cite[Lemma 2]{MR0185094}, we have $\Del_{z,k} \leq |z| / p \Del_{p, k}$.
  Thus, as $|z|/p \leq \operatorname{sec} \alpha$ throughout $C_\alpha$, we have
  \[
    \norm{ \sum_{k=m}^n \lambda_k^z } < \epsilon \left( \sum_{k=m}^{n-1} \Del_{p,k} + \lambda_n^p \right) = \epsilon \lambda_n^p < \epsilon \norm{T^p}.
    \]
\end{proof}

As in the classical case, the Jensen--Cahen theorem paves the way for a \emph{holomorphic} zeta function.

\begin{theorem}
  Let $0 \leq T \leq 1 \in \sch{p}{E_A}$.
  Then the map $z \mapsto \zeta(z, T)$ is holomorphic on the half-plane $\C_{\Re z > p} = \{z \mid \Re z > p\}$.
  Moreover, for all compact subsets $K \subset \C_{\Re z > p}$, all $x \in X$ and all $\epsilon > 0$, there is a neighbourhood $U$ of $x$ on which
  \[
    \sup_{z \in K} | \zeta(z,T)(y) - \zeta(z,T)(x) | < \epsilon
  \]
  for all $y \in U$.
\end{theorem}
\begin{proof}
 For the first statement we consider the $A$-valued function $\zeta_n(\cdot, T) \colon z \mapsto \sum_{k=0}^n \lambda_k^z$ on $\C_{\Re z > p}$. Recall that all bounded functionals $A \to \C$ decompose in four positive linear functionals.
By the Riesz representation theorem, all such positive linear functionals are given by positive, finite, regular Borel measures $\mu$ on $\gfdual{A}$ under the identification $\phi_\mu(f) \eqdef \int f d \mu$.
In particular, we have $\phi_\mu(\lambda_k^z) = \int \lambda_k(x)^z d \mu(x)$.
Now, if $C$ is a contour around $z_0$ in $\C_{\Re z > p}$, the contour integral $\oint_C \lambda_k(x)^z$ vanishes for all $x \in X$.
Note that $|\lambda_k^z| = \lambda_k^{\Re z}$ and so, by Fubini's theorem, $\oint_C \int \lambda_k(x)^z d \mu (x)= \int \left( \oint_C \lambda_k(x)^z \right) d \mu(x) = 0$. 
By Morera's theorem, we conclude that $\phi_\mu(\zeta_n)$ is holomorphic on $\C_{\Re z > p}$.
Moreover, as $\zeta_n(\cdot, T) \to \zeta(\cdot, T)$ uniformly on compact subsets of $\C_{\Re z > p}$ by Lemma~\ref{lemma:continuous-jensen-cahen}, the function $\phi_\mu(\zeta(\cdot, T))$ is holomorphic on $\C_{\Re z > p}$ as well.
By \cite[Theorem 3.31]{MR1157815}, we may conclude that $\zeta  (\cdot, T)$ is holomorphic, as an $A$-valued function, on the half-plane $\Re z > p$.

For the second statement, note that by Lemma~\ref{lemma:continuous-jensen-cahen}, for all $\epsilon > 0$ there is, for all compact subsets $K \subset \C_{\Re z > p}$, some $m_0$ with $\norm{\zeta_{m}(z, T) - \zeta(z, T)}_A < \epsilon$ for all $z \in K$ and all $m \geq m_0$.

Assume without loss of generality that $\lambda_k(x) > 0$ and pick a neighbourhood $U$ of $x$ on which $\lambda_k(y) > 0$, for all $k=1,\ldots, m$. For any $\epsilon_0>0$ let $V \subset U$ be such that $| \ln \lambda_k(x) - \ln \lambda_k(y) | < \epsilon_0$ for all $x, y \in V$.
Then, $\lambda_k(y)^z = e^{z \ln \lambda_k(y)}$ for all $z \in K$, so that
$| \lambda_k(x)^z - \lambda_k(y)^z | \leq | 1 - e^{zs} | |\lambda_k^{p}(x)|$ for some $s \in \C$ with $|s| < \epsilon_0$.
Moreover, $| 1 - e^{zs} | \leq | 1 - e^{\epsilon_1} | \leq \epsilon_1 e^{\epsilon_1}$, where $\epsilon_1 = \epsilon_0  \sup_{z \in K} |z|$.

If we now pick $\epsilon_0$ such that $\epsilon_1 e^{\epsilon_1} %\epsilon_0
|\lambda_k^p(x)| < \epsilon/m$, we conclude that for all $y \in V$ and all $z \in K$ we have $| \lambda_k(x)^z - \lambda_k(y)^z | < \epsilon/m$. Consequently, we find that $\| \zeta_m(\cdot, T)(x) - \zeta_m(\cdot, T)(y) \| < \epsilon$ as desired. 
%% This estimate extends for finite $m \geq m_0$ to the functions $\zeta_m(\cdot, T)$ on $K$, {\em i.e.} $\| \zeta_m(\cdot, T)(x) - \zeta_m(\cdot, T)(y) \| < \epsilon$. 
%% satisfy the same property for finite $m \geq m_0$, and therefore so does $\zeta(\cdot, T)$.
  \end{proof}

\begin{remark}
  It would be desirable to extend the previous Lemma and Theorem to the functions $\tr a T^{z}$, for $a \in \lnmp{E_A}$.
  However, since the spectral projections of $T$ are in general not even weakly continuous, this would be a nontrivial extension.
  We expect that a possibility for such an extension would be to investigate the functions $x \mapsto \tr p_k(x) a(x) p_k(x) / \operatorname{rank}(p_k(x))$, where $p_k(x)$ is the spectral projection on the eigenspace belonging to the eigenvalue $\lambda_k(x)$ of $T(x)$, and then use these expressions as coefficients in the continuous Dirichlet series.
%  An even stronger extension to the `non-compact fiber' case where $T$ does not lie in any Schatten class but only $aT^{p}$ does for certain $a \in \lnmp{E_A}$ would be desirable as well.
\end{remark}

The next step in the classical case would be to show that certain operators have zeta functions that can be continued meromorphically to all of $\C$, with a discrete set of poles. The residues at these poles then yields interesting information about the operator $T$. For instance, if $T$ is the bounded transform $(1 + D^2)^{-1/2}$ of a pseudodifferential operator these residues give geometric information about the pertinent background manifold ({\em cf.} \cite{MR2273508} for more details). %%  the case where
%% $T$ is the bounded transform $(1 + D^2)^{-1/2}$ of a pseudodifferential operator.
For this reason, it would be very desirable to have a reasonable criterion under which our zeta function of operators on Hilbert \cstar-modules can be continued meromorphically to all of $\C$, but further research in that direction is beyond the scope of the present work.

%% \begin{remark}
%%   Let $D$ be a regular, self-adjoint operator on $C_0(X,H)$ such that, for all $x$, the family $y \mapsto D(x) - D(y)$ of operators extends to an element of $C_b(\gfdual{A}, B(H))$.
%%   From Allahverdiev's theorem we obtain the singular value estimate $|\sigma_n(D(x)) - \sigma_n(D(y))| \leq \norm{D(x) - D(y)}$ for all $n$.
%%   Thus,  $\tr e^{-t D^2(x)} / \tr e^{-t D^2(y)} \leq e^{\norm{D(x) - D(y)}}$ for $t \leq 1$.
%%   Then, if $D(x)$ has the property that $\tr e^{-t D^2(x)}$ possesses an asymptotic expansion $\tr e^{-t D^2(x)} \sim \sum_{k=0}^\infty c_k(x) t^{\alpha_k}$, so does $D(y)$ and the $c_k$ are continuous.
%%   The zeta functions can then be extended meromorphically to $\C \setminus \{\alpha_k\}_k$, with continuous residues.
%%   We speculate that this statement can be extended to the case where $(D(x) - D(y))(D(x) - \lambda)^{-1}$ is bounded and norm continuous, cf. \cite[Theorem IV.3.18]{MR1335452}.
%%   See  Proposition~\ref{proposition:relative-bound-impies-summability}, below, for a similar condition being used to show that $D$ has Schatten class resolvent.
%% \end{remark}

\subsection{Summability of unbounded Kasparov cycles}
\label{sec:summ-some-unbo}
We now apply the just-developed theory of Schatten classes to arrive at a notion of summability for unbounded Kasparov cycles over a commutative $C^*$-algebra. This notion is supposed to generalize summability for spectral triples (as unbounded Kasparov cycles over $\C$) \cite[Section IV.2]{MR1303779} ({\em cf.} \cite[Definition 10.8]{MR1789831}) . 
We refer to \cite{MR715325,MR3213549,MR3107519,MR3545223} for all relevant notions of unbounded Kasparov cycles, external and internal Kasparov product, and to \cite{MR3873573,suijlekom2019immersions} for the specific application to Riemannian submersions and immersions to be discussed below. 

In order to set the notation, for $A,B$ two $C^*$-algebras we let $({}_A E_B,S)$ be an {\em unbounded Kasparov $A-B$ cycle}, consisting of
  \begin{itemize}
  \item $E_B$ is a (graded) Hilbert $B$-module
  \item $A$ is represented on $E_B$ by adjointable (even) operators.
  \item $S$ is a regular, self-adjoint (and odd) operator, densely defined on $\dom(S) \subset E_A$.
  \item $a(1+S^2)^{- 1}$ is in $\cptmp{E_B}$ for all $a \in A$.
  \item There exists a dense subalgebra $\mathcal{A} \subset A$ that preserves $\dom(S)$ and is such that for all $a \in \mathcal{A}$, the commutator $[S,a]$ extends to an adjointable operator on $E_B$.
  \end{itemize}

\begin{definition}
  Let $B$ be a unital, commutative \cstar-algebra and let $({}_A E_B, S)$ be an unbounded Kasparov $A-B$ cycle.
  We say that $({}_A E_B, S)$ is {\em $p$-summable} if
  \[
    (1 + S^2)^{-\frac12} \in \sch{q}{E_B}.
  \]

\end{definition}

For simplicity, we restrict to the case where $A$ is unital.
See \cite[Section 2]{MR3221983} for an understanding of the nonunital case.

\begin{example}
  If $(A, H, D)$ is a $p$-summable spectral triple ({\em cf.} \cite[Definition 10.8]{MR1789831}), then it is an unbounded $p$-summable Kasparov $(A, \C)$ cycle.
\end{example}
By the very definition of the Schatten classes, the localization of an unbounded $p$-summable Kasparov $(A, B)$-cycle along a character of $B$ yields a $p$-summable spectral triple. In the other direction, it is not true that if all such localized spectral triples are $p$-summable, then the original unbounded KK-cycle is $p$-summable: one needs the fiberwise $\mathcal L^p$-norms of the resolvents to be continuous over the base.

\subsubsection{Summability and the external Kasparov product}
One of the key results in \cite{MR715325} was an explicit and linear formula for the external Kasparov product. More precisely, they showed that two unbounded Kasparov cycles (restricting to the even-odd case for simplicity) $(E_{B},\gamma,S)$ and $(F_{C},T)$ can be combined into an {\em external product} unbounded KK-cycle over the minimal tensor product $B \otimes \C$:
$$
( (E \otimes F)_{B \otimes C} , S \otimes 1 + \gamma \otimes T ).
$$
For spectral triples this can be understood as the direct product of the corresponding (noncommutative) spaces. In any case, it is desirable to have an additive property of summability for this external product in the case of a commutative base. 

\begin{lemma} \label{lemma:bound-on-powers-of-commutating-resolvents}
  If $a, b$ are positive, (resolvent) commuting, regular operators on a Hilbert \cstar-module over a commutative $C^*$-algebra, then for $p, q > 0$ one has
  \[
    (1 + a + b)^{-p-q} \leq (1 + a)^{-p/2} (1 + b)^{-q} (1 + a)^{-p/2}
  \]
\end{lemma}
\begin{proof}
  By positivity of $a, b$, we have $(b + 1)^{-1} \geq (a + b + 1)^{-1} \leq (a + 1)^{-1}$, and by commutativity of the \cstar-algebra generated by the resolvents of $a, b$, we have $(a + b + 1)^{-p-q} \leq (a + 1)^{-p/2} (a + b + 1)^{-q} (a + 1)^{-p/2} \leq (a + 1)^{-p/2} (b + 1)^{-q} (a + 1)^{-p/2}$.
\end{proof}

\begin{corollary}
  The summability of unbounded Kasparov modules (over commutative $C^*$-algebras) is additive under the \emph{exterior} product.
\end{corollary}
\begin{proof}
  The corresponding selfadjoint operators $(S \otimes 1)$ and $(\gamma \otimes T)$ on $E \otimes F$ anticommute, and the actions of $B ,C$ commute.
  Thus, we have $|S \otimes 1 + \gamma \otimes T|^2 = |S \otimes 1|^2 + |\gamma \otimes T|^2$, which commute.
  Moreover, the exterior product $\{\fram_i \otimes \framt_j\}_{ij}$ of frames is a frame.
  We conclude, with the Lemma, that $|S \otimes 1 + 1 \otimes T + i|^{-p-q} \leq |S \otimes 1 + i|^{-p} |1 \otimes T + i|^{-q}$, and the $B \otimes C$-valued trace of the latter is just the tensor product of the traces of $|S + i|^{-p}$ and $|T + i|^{-q}$.
\end{proof}

Of course, the real challenge is to establish the compatibility of summability with the internal unbounded Kasparov product. Clearly, Lemma \ref{lemma:bound-on-powers-of-commutating-resolvents} is then not sufficient but we leave its (challenging) extension for future research. Instead, we limit ourselves to establishing such summability results for certain classes of geometric examples. When combined with \cite{MR3500818,MR3873573,vandendungenKasparovProductSubmersions2018,suijlekom2019immersions} where these examples were described in terms of unbounded Kasparov cycles, one may conclude the sought-for compatibility of summability with the internal product at least for these geometric examples.

\subsubsection{Example: Riemannian submersions}
In \cite{MR3873573} the factorization of the Dirac operator $D_Y$ on $Y$ in terms of a vertical operator $S$ and the Dirac operator $D_X$ on $X$ was studied for a Riemannian submersion $\pi \colon Y \to X$ of compact spin$^{c}$ manifolds (more general proper Riemannian submersions were considered in \cite{kaad2017factorization,vandendungenKasparovProductSubmersions2018,vandendungenLocalisationsHalfclosedModules2020}).

We let $L^2(\mathcal{S}_X)$ and $L^2(\mathcal{S}_Y)$ denote the Hilbert space completions of the spinor modules $\Gamma^\infty(\mathcal{S}_X)$ and $\Gamma^\infty(\mathcal{S}_Y)$, respectively. Based on a certain $C^\infty(Y)$-module of smooth sections of the vertical spinor bundle $\mathcal{S}_V$ one then defines a Hilbert $C^*$-module $E_{C(X)}$ between $C(Y)$ and $C(X)$, together with a self-adjoint and regular unbounded operator $D_V$ on $E$, such that 
\[
L^2(\mathcal{S}_Y) \cong E \widehat{\otimes}_{C(X)} L^2(\mathcal{S}_X)
\]
and in such a way that the operator $D_Y$ corresponds to the tensor sum $D_V \otimes \gamma_E + 1 \otimes_\nabla D_X$ for some metric connection $\nabla$ on $E_{C(X)}$ and the grading operator $\gamma_E$ on $E$ (up to an explicit error term related to the curvature).

Let us analyze here the summability aspects of the operator $D_V: \dom(D_V) \to E$. The main property that we will use below is that $D_V$ is the closure of a so-called {\em vertically elliptic operator} $\mathcal D: \Gamma^\infty(\mathcal{S}_V) \to \Gamma^\infty(\mathcal{S}_V)$. This means that for all $f \in C^\infty(Y)$, $[\mathcal D ,f]$ is an endomorphism of $\Gamma^\infty(\mathcal S_V)$ that is invertible at all points where $df|_{\ker d \pi}$ is nonzero (see \cite{MR3873573,kaad2017factorization} for more details). In fact, this allows one to prove \cite[Theorem 3]{kaad2017factorization} that the pair $(E,D_V)$ is an unbounded Kasparov $C(Y)-C(X)$ cycle.

As far as summability is concerned, note that the restrictions $\chi^* \mathcal D= \mathcal D_x$ of $\mathcal D$ to the fibers of $x \in X$ (for the character $\chi:C(X) \to \C$) are elliptic Dirac type operators. Since the dimension of the fibers is constant and equal to $\dim F$ for the typical fiber $F$, one has $\norm{(1 + D_x^2)^{-1/2}}_p < \infty$ for all $p > \dim F$ \cite{MR1441908} ({\em cf.} \cite[Theorem 11.1]{MR1789831}). The question, then, is whether this pointwise trace is continuous.

\begin{proposition}
  The $\mathcal L^p$-norm of $(1+ \mathcal D_x^2)^{-1/2}$ defines (as $x$ varies over $X$) a continuous function on $X$ for any $p > \dim F$. Consequently, the unbounded Kasparov $C(Y)-C(X)$ cycle $(E,D_V)$ is $p$-summable for all $p > \dim F$.
  \end{proposition}
\begin{proof}[Proof (based on {\cite[Section 2.3]{kaad2017factorization}})]
For simplicity of exposition, we will assume that the bundle $\mathcal S_Y \to Y$ is locally trivial over $X$; that is, we assume that around each point $x_0 \in X$ there exists a neighbourhood $U \subset X$ such that 1) there exists a diffeomorphism $\psi \colon \pi^{-1}(U) \to U \times F$, and 2) the bundle $\mathcal S_Y \to Y$ can be smoothly and unitarily trivialized over $U$.
If we pull sections of $\mathcal S_Y$ over $U$ back through this trivialization, we obtain a family of Dirac-type differential operators $\{ \mathcal D_x\}$ on the trivial bundle $\C^k$ over the compact Riemannian manifold $F$.

% $C^\infty_c(F)^{\oplus k}$ as in \cite[Section 2.3]{kaad2017factorization}.

In particular, these operators can be written as
$$
\mathcal D_x = \sum_{j=1}^{\dim F} A_j(x,z)  \frac{\partial }{\partial z_j} + B(x,z)
$$
with $A_j,B$ symmetric matrix-valued smooth functions on $U \times F$ and $A_j$ invertible.
Thus, there are smooth families $Z_{x, x'}(z) = A_j(x', z) A_j^{-1}(x, z)$, $W_{x, x'}(z) = B(x', z) - Z_{x, x'} B(x, z)$ of matrices such that $\mathcal D_{x'} = Z_{x, x'} \mathcal D_x + W_{x, x'}$, with in particular $\lim_{x' \to x} (\id - Z_{x, x'}) = \lim_{x' \to x} W_{x, x'} = 0$.

Now denote the closure of $\mathcal D_x$ by $D_x$.
Note that $D_x$ is selfadjoint by compactness of $F$. % Higson-Roe 10.2.6!
Moreover, because $\mathcal D_x$ is an elliptic differential operator of order 1, the resolvent $(D_x + i)^{-1}$ lies in $\sch{p}{L^2(F, \C^k)}$ for all $p > \dim F$.

Then, $D_x ( D_x + i)^{-1}$ being bounded, we see that $\lim_{x' \to x} \| (D_{x'} - D_x)(D_x + i)^{-1} \| \leq \lim_{x' \to x} \| \id - Z_{x, x'} \| + \|W_{x, x'}\| = 0$. % after all, \|D(1 + D^2)^{-1/2}\| \leq 1, see e.g. Lance, and \|(1 + D^2)\|^{-1/2}\| < 1 because D^2 is positive, see Kato p273 formula 3.17
Thus, by the resolvent identity, we conclude that $\lim_{x' \to x} \| (D_x + i)^{-1} - (D_{x'} + i)^{-1} \|_p = 0$ for all $p > \dim F$, and so $(1 + D^2)^{-1/2} \in \sch{p}{E}$ if and only if $p > \dim F$.
%
% for the real deal, you need a compactly supported partition of unity in the fiber so that you work locally in Y, not X
% realize D as the sum of a bunch of such operators
%
% Let, by compactness, $W \subset X$ be such that $\{U_i\}_{i=1}^n$ is a finite open cover of $\pi^{-1}(W)$ by open subsets of the type considered above and let $\eta_i$ be a smooth partition of unity subordinate to this open cover.
% Write $\widetilde{E}^W$ for the Hilbert $C_0(W)$-module $\cbstr(W, L^2(B_\delta(\R^d), \bigoplus_{i=1}^n \C^k))$.
%
% Then, the map $\alpha \colon E_{C_0(W)} \to \widetilde{E}^W$, sending $v \in \Gamma(\mathcal S_x|_W)$ to $b \mapsto (\alpha_{U_1}^{-1} (\eta_1 \cdot v)(b), \dotsc, \alpha_{U_n}^{-1}(\eta_n \cdot v)(b))$, is an isometric embedding with adjoint $\alpha^* \colon (f_1, \dotsc, f_n) \to \sum_{i=1}^n \eta_i \alpha_{U_i}(f_i)$.
% Moreover, with $\widetilde{D^W} \eqdef \bigoplus_i D^{U_i}$,  $D_V - \alpha^* \circ \widetilde{D^W} \circ \alpha$ is an endomorphism of the bundle $\mathcal S_Y|_W$ and is in particular bounded.
%
% By the analysis above,
% \[
% \lim_{b' \to b} \norm{(\widetilde{D^W}_x - \widetilde{D^W}_{x'})(1 + (\widetilde{D^W}_x)^* \widetilde{D^W}_x)^{-1/2}} = 0
% \]
% for $b \in W$, so that the resolvent of $\alpha^* \circ \widetilde{D^W} \circ \alpha$ is continuous in pointwise $\mathcal L^p$-norm for all $p > \dim F$.
% By the resolvent identity, we conclude that $(1 + D^2)^{-1/2} \in \sch{p}{E}$ if and onlyf $p > \dim F$.
\end{proof}

Note that when we combine this result with the factorization result \cite{MR3873573} of $D_Y$ in terms of $D_V$ and $D_X$ we obtain for Riemannian submersions of spin$^c$ manifolds the desired additivity of summability for the unbounded interior Kasparov product.

  \subsubsection{Example: Embedding spheres in Euclidean space}
  We consider a special class of immersions, given by the embedding of spheres $S^n$ in Euclidean space $\R^{n+1}$. This is based on \cite{MR775126,suijlekom2019immersions,verhoevenEmbeddingCirclePlane2019}. As in \cite{MR775126} the embedding $S^n \to \R^{n+1}$ gives rise to an immersion class in KK-theory. For spheres, the unbounded representative is given by the module $C_0(S^n \times (-\epsilon, \epsilon))$ based on a normal neighborhood of $S^n\subset \R^{n+1}$, equipped with the regular self-adjoint operator $S$ given by the multiplication operator with a suitable function $f: (-\epsilon, \epsilon ) \to \R$. For convenience, we will take $S$ to be multiplication by the function
  $$
f(s) = \frac{\pi}{2 \epsilon} \tan \left( \frac{\pi s}{2 \epsilon} \right); \qquad (s \in (-\epsilon, \epsilon)).
  $$
Since $(i+f)^{-1}$ is clearly a $C_0$-function on $S^n \times (-\epsilon, \epsilon)$, we find as in \cite[Lemma 2.3]{suijlekom2019immersions} that $(i+S)^{-1}$ is a compact operator on the Hilbert module $C_0(S^n \times (-\epsilon, \epsilon))$ and so forms an unbounded Kasparov $C(S^n)-C_0(\R^{n+1})$ cycle. 

\begin{proposition}
  We have $(1+S^2)^{-1/2} \in \mathcal L^p( C_0(S^n \times (-\epsilon, \epsilon)))$ for any $p> 0$. Hence the unbounded Kasparov $C(S^n)-C_0(\R^{n+1})$ is $p$-summable for all $p > 0$.
\end{proposition}
\proof
The pointwise localizations of the Hilbert $C_0(X)$-module $C_0(X)$ are one-dimensional, so that the pointwise $\mathcal L^p$-norm of $g \in \lnmp{C_0(X)} = C_b(X)$ is given by pointwise evaluation of $|g|$. Hence, $\tr (1+S^2)^{-p/2} = (1 + f^2)^{-p/2}$, which lies in $C_0$ for all $p > 0$.
%% Let us consider an approximate identity $\{ u_n \}$ for $C_0((-\epsilon, \epsilon))$ such that $u_n \leq u_{n+1}$, and define a frame $\{ \fram_n \}$ for $C_0(S^n \times (-\epsilon, \epsilon))$ such that $|\fram_n|^2 = u_{n+1}- u_{n}$. Then we have 
%% $$
%% \tr (1+S^2)^{-p/2} = \sum_n \langle \fram_n , (1+f^2)^{-p/2} \fram_n \rangle = \lim_n u_n (1+f^2)^{-p/2}= (1+f^2)^{-p/2}. 
%% $$
%% Since this is an element $C_0(S^n \times (-\epsilon, \epsilon)) \subseteq C_0(\R^{n+1})$ for any $p>0$, the result follows.
\endproof

Again, this is a confirmation of additivity for summability under the unbounded interior Kasparov product. Indeed, in \cite{suijlekom2019immersions} it was shown that $D_{S^n}$ can be related to the immersion class defined by $S$ as above and $D_{\R^{n+1}}$ in the following way. Namely, the unbounded interior product of $S$ and $D_{\R^{n+1}}$ is equal to the unbounded interior product of $D_{S^n}$ with a so-called index cycle $T$. The latter represents the identity at the bounded level but is in fact a $p$-summable Kasparov cycle for all $p > 1$.

\begin{proposition}
The selfadjoint closure $T$ of the operator
\[
T_0 = \begin{pmatrix} 0 & - i \partial_s - i f(s) \\ i \partial_s + i f(s) & 0 \end{pmatrix}
\]
on $C_c^\infty((- \epsilon, \epsilon), \C^2)$ is $p$-summable for all $p > 1$.
\end{proposition}
\begin{proof}
As in \cite[Lemma 2.11]{suijlekom2019immersions}, for $|\lambda| > \frac{\pi}{2 \epsilon}$ we have $\Delta_\epsilon + 1 < T^2 + \lambda^2 + 1$, where $\Delta_\epsilon$ is the closure of the Dirichlet Laplacian on $C_c^\infty((-\epsilon, \epsilon))^{\oplus 2}$.
On the other hand, if $\Delta_{\epsilon/2}$ is the closure of the Dirichlet Laplacian on $C_c^\infty((-\epsilon/2, \epsilon/2))^{\oplus 2}$ and $c = f^2(\epsilon/2) + |f'(\epsilon/2)|$, then $\norm{T^2 - \Delta_{\epsilon/2} |_{L^2((-\epsilon/2, \epsilon/2))^{\oplus 2}}} = c < \infty$ so that, by the min-max principle, the singular values $\sigma_n(T^2)$ are bounded from above by $\sigma_n(\Delta_{\epsilon/2}) + c$.
Thus, one has $\| (\Delta_{\epsilon/2} + \lambda^2 + c + 1)^{-1} \|_p \leq \| (T^2 +\lambda^2 + 1)^{-1} \|_p \leq \| (\Delta + 1)^{-1} \|_p$ and so $(T \pm \lambda i)^{-1} \in \mathcal L^p$ if and only if $p > 1$.
\end{proof}

 As such, the summability of $D_{\R^{n+1}}$ plus that of the immersion cycle ({\em i.e.} $0^+$) indeed coincides with the summability of $D_{S^n}$ plus that of the index cycle. 

% Note that the Kasparov product $D_x$ itself is n+1-summable:
% we have \nu_k^2 / (1 - \epsilon)^2 + \lambda_n^2 \geq \sigma_m(D_x^2) \leq \nu_k^2 / (1 - \epsilon)^2 + \lambda_n^2 whenever $m \geq k, n$ as in  proposition 3.12[loc.cit.]
  
  \subsubsection {Example: Actions of $\Z$}

%  Let $M$ be a compact Riemannian manifold.
  Consider the standard $C(X)$-module $E = L^2(S^1) \otimes_\C C(X)$ equipped with the `Dirac' operator
  \[
    D = D_{S^1} \otimes 1.
  \]
  Then $\tr \chi_* |D|^{-p} = \tr |D_{S^1}|^{-p} < \infty$ for all characters $\chi$ and $p> 1$ so that we have $|D|^{-1} \in \sch{p}{E}$ for $p>1$. 
  %  Thus, if a \cstar-algebra $A$ is represented on $E$, the same holds (a fortiori) for $a|D|^{-1-\epsilon}$, for all $a \in A$.
 
  As an example of an unbounded Kasparov cycle, consider a homeomorphism on $X$ and consider the action $n \cdot f \eqdef f \circ \phi^n$ of $\Z$ on $C(X)$.
  Let $C(X) \rtimes_\phi \Z$ be the corresponding full crossed product \cstar-algebra. %, with a dense subset $C(X) \Z$ consisting of finite sums $\sum_k f_k u_k$.
  Consider the unitary $U \in \lnmp{E}$ given by $U = S \otimes 1$, where $S$ is multiplication by $\theta \mapsto e^{i \pi \theta}$ on $L^2(S^1)$.
  Then, the map $\rho$ defined on finite sums in $C(X) \rtimes_\phi \Z$ by
  \[
  \rho \colon %C(X) \Z  \to \lnmp{E},
  \sum_k f_k u_k \mapsto \sum_k f_k U^k,
  \]
  where the left action of $C(X)$ is just given by pointwise multiplication, extends by universality to a representation of $C(X) \rtimes_\phi \Z$.
  Moreover, $[D, \sum_k f_k U^k] = \sum_k k f_k U^k$ because $[D_{S^1}, S^k] = k S$, so that there is a dense subalgebra of $C(X) \rtimes_{\phi} \Z$ with $[D, a]$ bounded.
  We conclude that $(E_{C(X)}, D)$ is an unbounded $p$-summable Kasparov $C(X) \rtimes_{\phi} \Z-C(X)$ cycle for all $p > 1$.

  In particular, one has
  \[
    \zeta_D(f u_k, z)(x) = \tr \chi_* (f U^k D^{-z}) = f(x) \tr_{L^2(S^1)} S^k |D_{S^1}|^{-z},
  \]
  so that $\zeta_D(\sum_k f_k u_k, z)$ extends meromorphically to $\C \setminus \{1\}$ and in fact
  \[
    \res_{z = 1} \zeta_D(\sum_k f_k u_k, z) \propto f_0.
  \]

\myacknowledgments{
We sincerely thank Jens Kaad for his advice during and following a month-long visit to Odense and for his valuable comments on the manuscript.

\nwoacknowledgment

}

\begingroup
\raggedright
%\RaggedRight %% needs package ragged2e
\printbibliography
\endgroup

\end{document}

% Local Variables:
% eval: (LaTeX-add-environments '("align*") '("example") '("theorem") '("lemma") '("proof") '("proposition") '("definition") '("corollary") '("remark") '("assumption"))
% End: